\documentclass[a4paper,10pt]{article}

\usepackage[affil-it]{authblk}

\usepackage[english]{babel}

\usepackage{mathrsfs}
\usepackage{amssymb}
\usepackage{amsmath}
\usepackage{latexsym}
\usepackage{amsthm}
\usepackage{amscd}
\usepackage{graphicx}

\usepackage[colorlinks=true]{hyperref}

\theoremstyle{plain}
\newtheorem{teo}{Theorem}[section]
\newtheorem{prop}[teo]{Proposition}
\newtheorem{coroll}[teo]{Corollary}
\newtheorem{lem}[teo]{Lemma}
\newtheorem{rmk}[teo]{Remark}

\theoremstyle{definition}
\newtheorem{defin}[teo]{Definition}

\numberwithin{equation}{section}

\newcommand{\R}{\mathbb R}
\newcommand{\Ha}{\mathcal H}
\newcommand{\Ll}{\mathcal L}
\newcommand{\Co}{\mathcal C}
\newcommand{\eps}{\varepsilon}

\newcommand{\Fc}{\mathcal F}

\newcommand{\Sg}{\mathcal{S}g}

\title{Approximation of sets of finite fractional perimeter by smooth sets
and comparison of local and global $s$-minimal surfaces}

\begin{document}

\author{Luca Lombardini}

\affil{\footnotesize Universit\`a degli Studi di Milano\\ Via Cesare Saldini, 50\\ 20133, Milano, Italia\\
E-mail address: luca.lombardini@unimi.it}

\date{}

\maketitle

\begin{abstract}
In the first part of this paper we show that\let\thefootnote\relax\footnotetext{Part of this work was carried out on the occasion of a visit to the School of Mathematics and Statistics of the University of Melbourne, which the author thanks
for the very warm hospitality.}
a set $E$ has locally finite\let\thefootnote\relax\footnotetext{I would also like to express my gratitude to Enrico Valdinoci
for his advice, support and patience.} $s$-perimeter if and only if it can be approximated in an appropriate sense
by smooth open sets.

In the second part we prove some elementary properties of local and global $s$-minimal sets, such as existence and compactness.\\
We also compare the two notions of minimizer (i.e. local and global), showing that in bounded open sets with Lipschitz boundary they coincide. However,
in general this is not true in unbounded open sets, where
a global $s$-minimal set may fail to exist (we provide an example in the case of a cylinder $\Omega\times\R$).

\end{abstract}

\tableofcontents

\begin{section}{Introduction and main results}

The aim of this paper
consists in better understanding the behavior of the family of sets having (locally) finite fractional perimeter.
In particular, we would like to show that this family is not ``too different'' from the family of Caccioppoli sets
(which are the sets having locally finite classical perimeter).

This paper somehow continues the 
study started in \cite{mine_fractal}. In particular, we showed there
(following an idea appeared in the seminal paper \cite{Visintin}) that sets having finite fractional perimeter can have a very rough boundary, which may indeed be a nowhere rectifiable fractal (like the von Koch snowflake).\\
This represents a dramatic difference between the fractional and the classical perimeter, since Caccioppoli sets have a ``big'' portion of the boundary, the so-called reduced boundary, which is $(n-1)$-rectifiable (by De Giorgi's structure Theorem).

Still, we prove in this paper that a set has (locally) finite fractional perimeter if and only if it can be approximated (in an appropriate way) by smooth open sets.
To be more precise, we show that a set $E$ has locally finite $s$-perimeter if and only if we can find a sequence of smooth open sets
which converge in measure to $E$,
whose boundaries converge to that of $E$ in a uniform sense, and whose $s$-perimeters converge to that of $E$ in every bounded open set.

Such a result is well known for Caccioppoli sets (see e.g.  \cite{Maggi}) and indeed this density property can be used to define the (classical) perimeter functional as the relaxation (with respect to $L^1_{loc}$ convergence) of the $\Ha^{n-1}$ measure of boundaries of smooth open sets, that is
\begin{equation*}\begin{split}P(E,\Omega)=\inf\Big\{\liminf_{k\to\infty}\Ha^{n-1}(\partial E_h&\cap\Omega)\,\big|\,
E_h\subset\R^n\textrm{ open with smooth}\\
&
\textrm{boundary, s.t. }E_h\xrightarrow{loc}E\Big\}.
\end{split}
\end{equation*}

%
%


\smallskip

The second part of this paper is concerned with sets minimizing the fractional perimeter.
The boundaries of these minimizers are often referred to as nonlocal minimal surfaces and naturally arise as
{\em limit interfaces} of long-range interaction phase transition models. In particular, in regimes where
the long-range interaction is dominant, the nonlocal Allen-Cahn energy functional $\Gamma$-converges
to the fractional perimeter (see \cite{SavEnAllenCahn})
and the minimal interfaces of the corresponding
Allen-Cahn equation approach locally uniformly
the nonlocal minimal surfaces (see \cite{SavEndensityest}).

We consider sets which are locally $s$-minimal in an open set $\Omega\subset\R^n$,
namely sets which
minimize the $s$-perimeter in every $\Omega'\subset\subset\Omega$,
and we prove existence and compactness results which extend those of \cite{CRS}.

We also compare this definition of local $s$-minimal set with
the definition of $s$-minimal set introduced in \cite{CRS}, proving that they coincide
when the domain $\Omega$ is a bounded open set with Lipschitz boundary (see Theorem \ref{confront_min_teo}).

In particular, the following existence results are proven:
\begin{itemize}
\item if $\Omega$ is an open set and $E_0$ is a fixed set, then there exists a set $E$ which is locally $s$-minimal in $\Omega$
and such that $E\setminus\Omega=E_0\setminus\Omega$;
\item there exist minimizers in the class of subgraphs, namely nonlocal nonparametric minimal surfaces (see Theorem \ref{nonparametric_exist_teo}
for a precise statement);
\item if $\Omega$ is an open set which has finite $s$-perimeter, then for every fixed set $E_0$ there exists
a set $E$ which is $s$-minimal in $\Omega$ and such that $E\setminus\Omega=E_0\setminus\Omega$.
\end{itemize}

On the other hand, we show that when the domain $\Omega$ is unbounded the nonlocal part
of the $s$-perimeter can be infinite, thus preventing the existence of
competitors having finite $s$-perimeter in $\Omega$ and hence also of ``global'' $s$-minimal sets.
In particular, we study this situation in a cylinder $\Omega^\infty:=\Omega\times\R\subset\R^{n+1}$,
considering as exterior data the subgraph of a (locally) bounded function.\\

In the next subsections we present the precise statements of the main results of this paper.
We begin by recalling the definition of fractional perimeter.

\begin{subsection}{Sets of (locally) finite $s$-perimeter}

Let $s\in(0,1)$ and let $\Omega\subset\mathbb R^n$ be an open set. The $s$-fractional perimeter
of a set $E\subset\mathbb R^n$ in $\Omega$ is defined as
\begin{equation*}
P_s(E,\Omega):=\mathcal L_s(E\cap\Omega,\Co E\cap\Omega)+
\mathcal L_s(E\cap\Omega,\Co E\setminus\Omega)+
\mathcal L_s(E\setminus\Omega,\Co E\cap\Omega),
\end{equation*}
where
\begin{equation*}
\mathcal L_s(A,B):=\int_A\int_B\frac{1}{|x-y|^{n+s}}\,dx\,dy,
\end{equation*}
for every couple of disjoint sets $A,\,B\subset\mathbb R^n$. We simply write $P_s(E)$ for $P_s(E,\R^n)$.

We say that a set $E\subset\R^n$ has locally finite $s$-perimeter in an open set $\Omega\subset\R^n$ if
\begin{equation}
P_s(E,\Omega')<\infty\qquad\textrm{for every open set }\Omega'\subset\subset\Omega.
\end{equation}

We remark that the family of sets having finite $s$-perimeter in $\Omega$ need not coincide with the family of sets of locally
finite $s$-perimeter in $\Omega$,
not even when $\Omega$
is ``nice'' (say bounded and with Lipschitz boundary). To be more precise, since
\begin{equation}\label{sup_for_per_eq}
P_s(E,\Omega)=\sup_{\Omega'\subset\subset\Omega}P_s(E,\Omega'),
\end{equation}
(see Proposition \ref{continuity_in_open_set_seq} and Remark \ref{continuity_in_open_set_seq_rmk}),
a set which has finite $s$-perimeter in $\Omega$ has also locally finite $s$-perimeter.
However the converse, in general, is false.\\
When $\Omega$ is not bounded it is clear that also for sets of locally finite $s$-perimeter
the sup in \eqref{sup_for_per_eq}
may be infinite (consider e.g. $\Omega=\R^n$ and $E=\{x_n\leq0\}$).

Actually, as shown in Remark \ref{counterex_loc_fin_per}, this may happen even when $\Omega$ is bounded and has Lipschitz boundary.
Roughly speaking, this is because the set $E$ might oscillate more and more as it approaches the boundary $\partial\Omega$.\\

\end{subsection}


\begin{subsection}{Approximation by smooth open sets}
We denote by $N_\rho(\Gamma)$ the $\rho$-neighborhood of a set $\Gamma\subset\R^n$, that is
\begin{equation*}
N_\rho(\Gamma):=\{x\in\R^n\,|\,d(x,\Gamma)<\rho\}.
\end{equation*}

The main approximation result is the following. In particular it shows that open sets with smooth boundary are dense in the family of sets of locally finite $s$-perimeter.

\begin{teo}\label{density_smooth_teo}
Let $\Omega\subset\R^n$ be an open set. A set $E\subset\R^n$ has locally finite $s$-perimeter in $\Omega$
if and only if there exists a sequence $E_h\subset\R^n$ of open sets with smooth boundary and $\eps_h\longrightarrow0^+$ such that
\begin{equation*}\begin{split}
& (i)\quad E_h\xrightarrow{loc}E,\qquad\sup_{h\in\mathbb N}P_s(E_h,\Omega')<\infty\quad\textrm{for every }\Omega'\subset\subset\Omega,\\
& (ii)\quad\lim_{h\to\infty}P_s(E_h,\Omega')=P_s(E,\Omega')\quad\textrm{for every }\Omega'\subset\subset\Omega,\\
& (iii)\quad\partial E_h\subset N_{\eps_h}(\partial E).
\end{split}
\end{equation*}
Moreover, if $\Omega=\R^n$ and the set $E$ is such that $|E|<\infty$ and $P_s(E)<\infty$, then
\begin{equation}\label{last_eq_teo_app}
E_h\longrightarrow E,\qquad\quad\qquad \lim_{h\to\infty}P_s(E_h)=P_s(E),
\end{equation}
and we can require each set $E_h$ to be bounded (instead of asking $(iii)$).
\end{teo}

The scheme of the proof is the following.

First of all, in Section 3.1 we prove appropriate approximation results for the functional
\begin{equation*}
\Fc(u,\Omega)=\frac{1}{2}\int_{\R^{2n}\setminus(\Co\Omega)^2}\frac{|u(x)-u(y)|}{|x-y|^{n+s}}\,dx\,dy,
\end{equation*}
which we believe might be interesting on their own.

Then we exploit the generalized coarea formula
\begin{equation*}
\Fc(u,\Omega)=\int_{-\infty}^\infty P_s(\{u>t\},\Omega)\,dt,
\end{equation*}
and Sard's Theorem to obtain the approximation of the set $E$ by superlevel sets of smooth functions which approximate $\chi_E$.

Finally, a diagonal argument guarantees the convergence of the $s$-perimeters
in every open set $\Omega'\subset\subset\Omega$.

\begin{rmk}
Let $\Omega\subset\R^n$ be a bounded open set with Lipschitz boundary and consider a set $E$ which has
finite $s$-perimeter in $\Omega$. Notice that if we apply Theorem \ref{density_smooth_teo},
in point $(ii)$ we do not get the convergence of the $s$-perimeters in $\Omega$, but only
in every $\Omega'\subset\subset\Omega$.
On the other hand, if we can find an open set $\mathcal O$ such that $\Omega\subset\subset\mathcal O$
and
\[P_s(E,\mathcal O)<\infty,\]
then we can apply Theorem \ref{density_smooth_teo} in $\mathcal O$. In particular, since $\Omega\subset\subset\mathcal O$,
by point $(ii)$ we obtain
\begin{equation}\label{conv_whole_Omega}
\lim_{h\to\infty}P_s(E_h,\Omega)=P_s(E,\Omega).
\end{equation}
\end{rmk}

Still, when $\Omega$ is a bounded open set with Lipschitz boundary, we can always obtain the convergence \eqref{conv_whole_Omega}
at the cost of weakening a little our request on the uniform convergence of the boundaries.

\begin{teo}\label{appro_in_bded_open}
Let $\Omega\subset\R^n$ be a bounded open set with Lipschitz boundary. A set $E\subset\R^n$ has finite $s$-perimeter
in $\Omega$ if and only if there exists a sequence $\{E_h\}$ of open sets with smooth boundary
and $\eps_h\longrightarrow0^+$ such that
\begin{equation*}\begin{split}
& (i)\quad E_h\xrightarrow{loc}E,\qquad\sup_{h\in\mathbb N}P_s(E_h,\Omega)<\infty,\\
& (ii)\quad\lim_{h\to\infty}P_s(E_h,\Omega)=P_s(E,\Omega),\\
& (iii)\quad\partial E_h\setminus N_{\eps_h}(\partial\Omega)\subset N_{\eps_h}(\partial E).
\end{split}
\end{equation*}
\end{teo}

Notice that in point $(iii)$ we do not ask the convergence of the boundaries in the whole of $\R^n$
but only in $\R^n\setminus N_\delta(\partial\Omega)$ (for any fixed $\delta>0$).
Since
$N_{\eps_h}(\partial\Omega)\searrow\partial\Omega$,
roughly speaking, the convergence holds in $\R^n$ ``in the limit''.

Moreover, we remark that point $(ii)$ in Theorem \ref{appro_in_bded_open}
guarantees the convergence of the $s$-perimeters also in every $\Omega'\subset\subset\Omega$ (see Remark \ref{rmk_conv_every_cpt_subopen}).

Finally, from the lower semicontinuity of the $s$-perimeter and Theorem \ref{appro_in_bded_open}, we obtain
\begin{coroll}
Let $\Omega\subset\R^n$ be a bounded open set with Lipschitz boundary and let $E\subset\R^n$. Then
\begin{equation}\begin{split}
P_s(E,\Omega)=\inf\Big\{\liminf_{h\to\infty}P_s(E_h&,\Omega)\,\big|\,E_h\subset\R^n\textrm{ open with smooth}\\
&
\textrm{boundary, s.t. }E_h\xrightarrow{loc}E\Big\}.
\end{split}
\end{equation}

\end{coroll}

For similar approximation results see also \cite{BVquantitative} and \cite{Planelike}.

\end{subsection}


\begin{subsection}{Nonlocal minimal surfaces}

First of all we give the definition of (locally) $s$-minimal sets.

\begin{defin}
Let $\Omega\subset\R^n$ be an open set and let $s\in(0,1)$. We say that a set $E\subset\R^n$
is $s$-minimal in $\Omega$ if $P_s(E,\Omega)<\infty$ and
\begin{equation*}
F\setminus\Omega=E\setminus\Omega\quad\Longrightarrow\quad P_s(E,\Omega)\leq P_s(F,\Omega).
\end{equation*}
We say that a set $E\subset\R^n$ is locally $s$-minimal in $\Omega$ if it is $s$-minimal in every $\Omega'\subset\subset\Omega$.
\end{defin}

When the open set $\Omega\subset\R^n$ is bounded and has Lipschitz boundary, the notions of $s$-minimal set and locally $s$-minimal set coincide.

\begin{teo}\label{confront_min_teo}
Let $\Omega\subset\R^n$ be a bounded open set with Lipschitz boundary and let $E\subset\R^n$. The following are equivalent

$\quad(i)$ $E$ is $s$-minimal in $\Omega$;

$\quad(ii)$ $P_s(E,\Omega)<\infty$ and
\begin{equation*}
P_s(E,\Omega)\leq P_s(F,\Omega)\qquad
\textrm{for every }F\subset\R^n\quad\textrm{s.t.}\quad E\Delta F\subset\subset\Omega;
\end{equation*}

$\quad(iii)$ $E$ is locally $s$-minimal in $\Omega$.

\end{teo}

 We remark that a set as in $(ii)$ is called a local minimizer for $P_s(-,\Omega)$ in \cite{Gamma} and a ``nonlocal area minimizing surface'' in $\Omega$ in \cite{DelPino}.

\begin{rmk}
The implications $(i)\Longrightarrow(ii)\Longrightarrow(iii)$ actually hold in any open set $\Omega\subset\R^n$.
\end{rmk}

In \cite{CRS} the authors proved that if $\Omega$ is a bounded open set with Lipschitz boundary,
then given any fixed set $E_0\subset\R^n$ we can find a set $E$ which is $s$-minimal in $\Omega$
and such that $E\setminus\Omega=E_0\setminus\Omega$.

This is because
\begin{equation*}
P_s(E_0\setminus\Omega,\Omega)\leq P_s(\Omega)<\infty,
\end{equation*}
so the exterior datum $E_0\setminus\Omega$ is itself an admissible competitor with finite $s$-perimeter in $\Omega$ and we can use
the direct method of the Calculus of Variations to obtain a minimizer.

In Section 2.3 we prove a compactness property which we use in Section 4.3 to
prove the following existence results, which extend that of \cite{CRS}.

\begin{teo}\label{glob_min_exist}
Let $\Omega\subset\R^n$ be an open set and let $E_0\subset\R^n$. Then there exists
a set $E\subset\R^n$ $s$-minimal in $\Omega$, with $E\setminus\Omega=E_0\setminus\Omega$, if and only if
there exists a set $F\subset\R^n$, with $F\setminus\Omega=E_0\setminus\Omega$ and such that
$
P_s(F,\Omega)<\infty.
$
\end{teo}

An immediate consequence of this Theorem is the existence of $s$-minimal sets in open sets having finite $s$-perimeter.
\begin{coroll}
Let $s\in(0,1)$ and let $\Omega\subset\R^n$ be an open set such that
\[P_s(\Omega)<\infty.\]
Then for every $E_0\subset\R^n$
there exists a set $E\subset\R^n$ $s$-minimal in $\Omega$, with $E\setminus\Omega=E_0\setminus\Omega$.
\end{coroll}

Even if we cannot find a competitor with finite $s$-perimeter, we can always find a locally $s$-minimal set.

\begin{coroll}\label{loc_min_set_cor}
Let $\Omega\subset\R^n$ be an open set and let $E_0\subset\R^n$.
Then there exists a set $E\subset\R^n$ locally $s$-minimal in $\Omega$, with $E\setminus\Omega=E_0\setminus\Omega$.
\end{coroll}

In Section 4.2 we also prove compactness results for (locally) $s$-minimal sets
(by slightly modifying the proof of Theorem 3.3 of \cite{CRS}, which proved compactness for $s$-minimal sets in a ball).
Namely, we prove that every limit set of a sequence of (locally) $s$-minimal sets is itself (locally) $s$-minimal.

\begin{teo}\label{minimal_comp}
Let $\Omega\subset\R^n$ be a bounded open set with Lipschitz boundary. Let $\{E_k\}$ be a sequence of $s$-minimal sets in $\Omega$, with $E_k\xrightarrow{loc}E$. Then $E$ is $s$-minimal in $\Omega$ and
\begin{equation}\label{conv_perimeter}
P_s(E,\Omega)=\lim_{k\to\infty}P_s(E_k,\Omega).
\end{equation}

\end{teo}

\begin{coroll}\label{local_minima_comp}
Let $\Omega\subset\R^n$ be an open set. Let $\{E_h\}$ be a sequence of sets locally $s$-minimal
in $\Omega$, with $E_h\xrightarrow{loc}E$. Then $E$ is locally $s$-minimal in $\Omega$ and
\begin{equation}\label{conv_perimeter_locally}
P_s(E,\Omega')=\lim_{h\to\infty}P_s(E_h,\Omega'),\qquad\textrm{for every }\Omega'\subset\subset\Omega.
\end{equation}

\end{coroll}

\begin{subsubsection}{Minimal sets in cylinders}

We have seen that a locally $s$-minimal set always exists, no matter what the domain $\Omega$
or the exterior data $E_0\setminus\Omega$ are.

On the other hand, the only requirement needed for the existence of an $s$-minimal set
is the existence of a competitor with finite $s$-perimeter.\\
We show that even in the case of a regular domain, like the cylinder $\Omega^\infty:=\Omega\times\R$,
with $\Omega\subset\R^n$ bounded with $C^{1,1}$ boundary,
such a competitor might not exist. Roughly speaking,
this is a consequence of the unboundedness of the domain $\Omega^\infty$, which forces the nonlocal part of the $s$-perimeter
to be infinite.

In Section 4.4 we study (locally) $s$-minimal sets in $\Omega^\infty$, with respect to the exterior data
given by the subgraph of a function $v$, that is
\[\Sg(v)=\{(x,t)\,|\,t<v(x)\}.\]

In particular, we consider sets which are $s$-minimal in the ``truncated'' cylinders $\Omega^k:=\Omega\times(-k,k)$,
showing that if the function $v$ is locally bounded, then these $s$-minimal sets cannot ``oscillate'' too much.
Namely their boundaries are constrained in a cylinder $\Omega\times(-M,M)$
independently on $k$.

As a consequence, we can find $k_0$ big enough such that a set $E$ is locally $s$-minimal in $\Omega^\infty$
if and only if it is $s$-minimal in $\Omega^{k_0}$ (see Lemma $\ref{bded_cyl_prop}$ and Proposition $\ref{bded_cyl_coroll}$ for the precise statements).

However, in general a set $s$-minimal in $\Omega^\infty$ does not exist. As an example we prove that
there cannot exist an $s$-minimal set having as exterior data the subgraph of a bounded function.\\

Frst of all, we remark that we can write the fractional perimeter as the sum
\begin{equation*}
P_s(E,\Omega)=P_s^L(E,\Omega)+P_s^{NL}(E,\Omega),
\end{equation*}
where
\begin{equation*}\begin{split}
&P_s^L(E,\Omega):=\mathcal L_s(E\cap\Omega,\Co E\cap\Omega)=\frac{1}{2}[\chi_E]_{W^{s,1}(\Omega)},\\
&
P_s^{NL}(E,\Omega):=\Ll_s(E\cap\Omega,\Co E\setminus\Omega)+\Ll_s(E\setminus\Omega,\Co E\cap\Omega).
\end{split}\end{equation*}
We can think of $P^L_s(E,\Omega)$ as the local part of the fractional perimeter, in the sense that if $|(E\Delta F)\cap\Omega|=0$,
then $P^L_s(F,\Omega)=P^L_s(E,\Omega)$.\\

The main result of Section 4.4 is the following

\begin{teo}\label{bound_unbound_per_cyl_prop}
Let $\Omega\subset\R^n$ be a bounded open set. Let $E\subset\R^{n+1}$ be such that
\begin{equation}\label{bound_hp_forml_subgraph}
\Omega\times(-\infty,-k]\subset E\cap\Omega^\infty\subset\Omega\times(-\infty,k],
\end{equation}
for some $k\in\mathbb N$, and suppose that $P_s(E,\Omega^{k+1})<\infty$. Then
\begin{equation*}
P_s^L(E,\Omega^\infty)<\infty.
\end{equation*}
On the other hand, if
\begin{equation}\label{bound_hp_forml_subgraph2}
\{x_{n+1}\leq-k\}\subset E\subset\{x_{n+1}\leq k\},
\end{equation}
then
\begin{equation*}
P^{NL}_s(E,\Omega^\infty)=\infty.
\end{equation*}
In particular, if $\Omega$ has $C^{1,1}$ boundary and $v\in L^\infty(\R^n)$, there cannot exist an $s$-minimal set in $\Omega^\infty$ with exterior data
\[
\Sg(v)\setminus\Omega^\infty=\{(x,t)\in\R^{n+1}\,|\,x\in\Co\Omega,\quad t<v(x)\}.\]
\end{teo}

\begin{rmk}
From Theorem $\ref{glob_min_exist}$ we see that if $v\in L^\infty(\R^n)$,
there cannot exist a set $E\subset\R^{n+1}$ such that $E\setminus\Omega^\infty=\Sg(v)\setminus\Omega^\infty$
and $P_s(E,\Omega^\infty)<\infty$.
\end{rmk}

As a consequence of the computations developed in the proof of Theorem $\ref{bound_unbound_per_cyl_prop}$,
in the end of Section 4.4 we also show that we cannot define a ``naive'' fractional nonlocal version
of the area functional as
\begin{equation*}
\mathcal A_s(u,\Omega):=P_s(\Sg(u),\Omega^\infty),
\end{equation*}
since this would be infinite even for very regular functions.\\

To conclude, we remark that as an immediate consequence of Corollary \ref{loc_min_set_cor}
and Theorem 1.1 in \cite{graph}, we obtain an existence result for the Plateau's problem in the class of subgraphs.

\begin{teo}\label{nonparametric_exist_teo}
Let $\Omega\subset\R^n$ be a bounded open set with $C^{1,1}$ boundary.
For every function $v\in C(\R^n)$ there exists a function
$u\in C(\overline{\Omega})$ such that, if
\begin{equation*}
\tilde{u}:=\chi_\Omega u+(1-\chi_\Omega)v,
\end{equation*}
then $\Sg(\tilde{u})$ is locally $s$-minimal in $\Omega^\infty$.
\end{teo}

Notice that, as remarked in \cite{graph}, the function $\tilde{u}$ need not be continuous. Indeed, because of boundary stickiness effects
of $s$-minimal surfaces (see e.g. \cite{boundary}), in general we might have
\[u_{|_{\partial\Omega}}\not=v_{|_{\partial\Omega}}.\]

\end{subsubsection}

\end{subsection}

\begin{subsection}{Notation and assumptions}

\begin{itemize}


\item We write $A\subset\subset B$ to mean that the closure of $A$ is compact and $\overline{A}\subset B$.

\item In $\R^n$ we will usually write
$|E|=\mathcal{L}^n(E)$ for the $n$-dimensional Lebesgue measure of a set $E\subset\R^n$.

\item By $A_h\xrightarrow{loc}A$ we mean that $\chi_{A_h}\longrightarrow\chi_A$ in $L^1_{loc}(\R^n)$,
i.e. for every bounded open set $\Omega\subset\R^n$ we have $|(A_h\Delta A)\cap\Omega|\longrightarrow0$.

\item We write $\Ha^d$ for the $d$-dimensional Hausdorff measure, for any $d\geq0$.

\item We define the dimensional constants
\begin{equation*}
\omega_d:=\frac{\pi^\frac{d}{2}}{\Gamma\big(\frac{d}{2}+1\big)},\qquad d\geq0.
\end{equation*}
In particular, we remark that $\omega_k=\mathcal{L}^k(B_1)$ is the volume of the $k$-dimensional unit ball $B_1\subset\R^k$
and $k\,\omega_k=\Ha^{k-1}(\mathbb{S}^{k-1})$ is the surface area of the $(k-1)$-dimensional sphere
\begin{equation*}
\mathbb{S}^{k-1}=\partial B_1=\{x\in\R^k\,|\,|x|=1\}.
\end{equation*}

\item Since
\begin{equation*}
|E\Delta F|=0\quad\Longrightarrow\quad
P_s(E,\Omega)=P_s(F,\Omega),
\end{equation*}
we can and will implicitly identify sets up to sets of zero measure.\\
In particular, equality and inclusions of sets will usually be considered in the measure sense, e.g. $E=F$ will usually mean
$|E\Delta F|=0$.\\
Moreover, whenever needed we will implicitly choose a particular representative for the class of $\chi_E$ in $L^1_{loc}(\R^n)$, as in the Remark below.

\end{itemize}

\begin{rmk}\label{gmt_assumption}
Let $E\subset\R^n$. Up to modifying $E$ on a set of measure zero, we can assume (see e.g. Appendix C of \cite{mine_fractal}) that
$E$ contains the measure theoretic interior
\begin{equation*}
E_1:=\{x\in\R^n\,|\,\exists\,r>0\textrm{ s.t. }|E\cap B_r(x)|=\omega_nr^n\}\subset E,
\end{equation*}
the complementary $\Co E$ contains its measure theoretic interior
\begin{equation*}
E_0:=\{x\in\R^n\,|\,\exists\,r>0\textrm{ s.t. }|E\cap B_r(x)|=0\}\subset\Co E,
\end{equation*}
and the topological boundary of $E$ coincides with its measure theoretic boundary, $\partial E=\partial^-E$,
where
\begin{equation*}\begin{split}
\partial^-E&:=\R^n\setminus(E_0\cup E_1)\\
&
=\{x\in\R^n\,|\,0<|E\cap B_r(x)|<\omega_nr^n\textrm{ for every }r>0\}.
\end{split}
\end{equation*}

\end{rmk}

\end{subsection}

\end{section}

\begin{section}{Tools}

It is convenient to point out the following easy but useful result.
\begin{prop}\label{subopensets}
Let $\Omega'\subset\Omega\subset\R^n$ be open sets and let $E\subset\R^n$. Then
\begin{equation}\begin{split}
P_s(E,\Omega)=P_s(E,\Omega')&+\Ll_s(E\cap(\Omega\setminus\Omega'),\Co E\setminus\Omega')\\
&
\qquad
+\Ll_s(E\setminus\Omega',\Co E\cap(\Omega\setminus\Omega')).
\end{split}
\end{equation}

As a consequence,

$(i)\quad$ if $E\subset\Omega$, then
\begin{equation*}
P_s(E,\Omega)=P_s(E),
\end{equation*}

$(ii)\quad$ if $E,\,F\subset\R^n$ have finite $s$-perimeter in $\Omega$
and $E\Delta F\subset\Omega'\subset\Omega$, then
\begin{equation}\label{difference_per_in_subsets}
P_s(E,\Omega)-P_s(F,\Omega)=P_s(E,\Omega')-P_s(F,\Omega').
\end{equation}

\end{prop}

\begin{rmk}\label{bded_set_frac_per}
In particular, if $E$ has finite $s$-perimeter in $\Omega$, then it has finite $s$-perimeter also in every open set $\Omega'\subset\Omega$.

\end{rmk}

\begin{subsection}{Bounded open sets with Lipschitz boundary}

Given $\emptyset\not=E\subset\R^n$, the distance function from $E$ is defined as
\begin{equation*}
d_E(x)=d(x,E):=\inf_{y\in E}|x-y|,\qquad\textrm{for }x\in\R^n.
\end{equation*}
The signed distance function from $\partial E$, negative inside $E$, is then defined as
\begin{equation}
\bar{d}_E(x)=\bar{d}(x,E):=d(x,E)-d(x,\Co E).
\end{equation}

We also define for every $r\in\R$ the sets
\begin{equation*}
E_r:=\{x\in\R^n\,|\,\bar{d}_E(x)<r\}.
\end{equation*}
Notice that if $\rho>0$, then
\begin{equation*}
N_\rho(\partial\Omega)=\{|\bar{d}_\Omega|<\rho\}=\Omega_\rho\setminus\overline{\Omega_{-\rho}}
\end{equation*}
is the $\rho$-tubular neighborhood of $\partial\Omega$.

Let $\Omega\subset\R^n$ be a bounded open set with Lipschitz boundary.
It is well known (see e.g.
Theorem 4.1 of \cite{LipApprox}) that also the bounded open sets $\Omega_r$ have
Lipschitz boundary, when $r$ is small enough, say $|r|<r_0$.\\
Notice that
\begin{equation*}
\partial\Omega_r=\{\bar{d}_\Omega=r\}.
\end{equation*}

Moreover the perimeter of $\Omega_r$
can be bounded uniformly in $r\in(-r_0,r_0)$ (see also Appendix B of \cite{mine_fractal} for a more detailed discussion)

\begin{prop}\label{bound_perimeter_unif}
Let $\Omega\subset\R^n$ be a bounded open set with Lipschitz boundary. Then there exists $r_0>0$ such that
$\Omega_r$ is a bounded open set with Lipschitz boundary for every $r\in(-r_0,r_0)$ and
\begin{equation}\label{bound_perimeter_unif_eq}
\sup_{|r|<r_0}\Ha^{n-1}(\{\bar{d}_\Omega=r\})<\infty.
\end{equation}
\end{prop}

As a consequence, exploiting the embedding $BV(\R^n)\hookrightarrow W^{s,1}(\R^n)$ we obtain
a uniform bound for the (global) $s$-perimeters
of the sets $\Omega_r$ (see Corollary 1.2 of \cite{mine_fractal})
\begin{coroll}
Let $\Omega\subset\R^n$ be a bounded open set with Lipschitz boundary. Then there exists $r_0>0$
such that
\begin{equation}\label{unif_bound_lip_frac_per}
\sup_{|r|<r_0}P_s(\Omega_r)<\infty.
\end{equation}
\end{coroll}

\begin{subsubsection}{Increasing sequences}

In particular, Proposition \ref{bound_perimeter_unif} shows that if $\Omega$ is a bounded open set with Lpschitz boundary,
then we can approximate it strictly from the inside with a sequence of bounded open sets
$\Omega_k:=\Omega_{-1/k}\subset\subset\Omega$. Moreover,
\eqref{bound_perimeter_unif_eq} gives a uniform bound on the measure of the boundaries
of the approximating sets.

Now we prove that any open set $\Omega\not=\emptyset$ can be approximated strictly from the inside with
a sequence of bounded open sets with smooth boundaries.

\begin{prop}\label{first_approx_prop}
Let $\Omega\subset\R^n$ be a bounded open set. For every $\eps>0$ there exists a bounded open set $\mathcal O_\eps\subset\R^n$ with smooth boundary, such that
\begin{equation}
\mathcal O_\eps\subset\subset\Omega\qquad\textrm{and}\qquad\partial\mathcal O_\eps\subset N_\eps(\partial\Omega).
\end{equation}

\begin{proof}

We show that we can approximate the set $\Omega_{-\eps/2}$
with a bounded open set $\mathcal O_\eps$ with smooth boundary such that $\partial\mathcal O_\eps\subset N_{\eps/4}(\partial\Omega_{-\eps/2})$.\\
In general $\mathcal O_\eps\not\subset\Omega_{-\eps/2}$.
However
\begin{equation}\label{eq_app_op}
\mathcal O_\eps\subset N_{\eps/4}(\Omega_{-\eps/2})\subset\subset\Omega
\quad\textrm{and indeed}\quad\Omega_{-3\eps/4}\subset\mathcal O_\eps\subset\Omega_{-\eps/4},
\end{equation}
proving the claim.

Let $u:=\chi_{\Omega_{-\eps/2}}$ and consider the regularized function
\begin{equation*}
v:=u_{\eps/4}=u\ast\eta_{\eps/4}
\end{equation*}
 (see Section 3 for the details about the mollifier $\eta$).
Since $v\in C^\infty(\R^n)$, we know from Sard's Theorem that the superlevel set $\{v>t\}$ is an open set with smooth boundary for a.e. $t\in(0,1)$.
Moreover notice that $0\leq v\leq1$, with
\begin{equation*}
\textrm{supp }v\subset N_{\eps/4}(\textrm{supp }u)=N_{\eps/4}(\Omega_{-\eps/2})\subset\Omega_{-\eps/4},
\end{equation*}
and
\begin{equation*}
v(x)=1\qquad\textrm{for every }x\in\Big\{y\in\Omega_{-\eps/2}\,\big|\,d(y,\partial\Omega_{-\eps/2})>\frac{\eps}{4}\Big\}
\supset\Omega_{-\frac{3}{4}\eps}.
\end{equation*}
This shows that $\mathcal O_\eps:=\{v>t\}$ (for any ``regular'' $t$) satisfies
$(\ref{eq_app_op})$.
\end{proof}
\end{prop}

\begin{coroll}\label{regular_approx_open_sets_coroll}
Let $\Omega\subset\R^n$ be an open set. Then there exists a sequence $\{\Omega_k\}$ of bounded open sets with smooth boundary such that $\Omega_k\nearrow\Omega$ strictly, i.e.
\begin{equation}
\Omega_k\subset\subset\Omega_{k+1}\subset\subset\Omega\qquad\textrm{and}\qquad\bigcup_{k\in\mathbb N}\Omega_k=\Omega.
\end{equation}
In particular $\Omega_k\xrightarrow{loc}\Omega$.

\begin{proof}

It is enough to notice that we can approximate $\Omega$ strictly from the inside with bounded open sets $\mathcal O_k\subset\R^n$, that is
\begin{equation*}
\mathcal O_k\subset\subset\mathcal O_{k+1}\subset\subset\Omega\qquad\textrm{and}\qquad\bigcup_{k\in\mathbb N}\mathcal O_k=\Omega.
\end{equation*}

Then we can exploit Proposition \ref{first_approx_prop}, and in particular $(\ref{eq_app_op})$, to find bounded open sets $\Omega_k\subset\R^n$ with smooth boundary such that
\begin{equation*}
\mathcal O_k\subset\subset\Omega_k\subset\subset\mathcal O_{k+1}.
\end{equation*}
Indeed we can take as $\Omega_k$ a set $\mathcal O_\eps$ corresponding to $\mathcal O_{k+1}$, with $\eps$
small enough to guarantee $\mathcal O_k\subset\subset\mathcal O_\eps$.\\
As for the sets $\mathcal O_k$, if $\Omega$ is bounded we can simply take
$
\mathcal O_k:=\Omega_{-2^{-k}}.
$
If $\Omega$ is not bounded, we can consider the sets $\Omega\cap B_{2^k}$ and define
\begin{equation*}
\mathcal O_k:=\big\{x\in\Omega\cap B_{2^k}\,|\,d\big(x,\partial(\Omega\cap B_{2^k})\big)>2^{-k}\big\}.
\end{equation*}
To conclude, notice that we have $\chi_{\Omega_k}\longrightarrow\chi_\Omega$ pointwise everywhere in $\R^n$, which implies the convergence in $L^1_{loc}(\R^n)$.
\end{proof}
\end{coroll}

\end{subsubsection}

\begin{subsubsection}{Some uniform estimates for $\rho$-neighborhoods}

The uniform bound $(\ref{bound_perimeter_unif_eq})$ on the perimeters of the sets $\Omega_\delta$
allows us to obtain the following estimates, which will be used in the sequel

\begin{lem}
Let $\Omega\subset\R^n$ be a bounded open set with Lipschitz boundary. Let $\delta\in(0,r_0)$. Then
\begin{equation}\label{uniform_bound_strips}\begin{split}
&(i)\quad\Ll_s(\Omega_{-\delta},\Omega\setminus\Omega_{-\delta})\leq C\,\delta^{1-s},\\
&
(ii)\quad\Ll_s(\Omega,\Omega_\delta\setminus\Omega)\leq C\,\delta^{1-s}
\quad\textrm{and}
\quad\Ll_s(\Omega\setminus\Omega_{-\delta},\Co\Omega)\leq C\,\delta^{1-s},
\end{split}\end{equation}
where the constant $C$ is
\begin{equation*}
C:=\frac{n\omega_n}{s(1-s)}\,\sup_{|r|<r_0}\Ha^{n-1}(\{\bar{d}_\Omega=r\}).
\end{equation*}

\begin{proof}

By using the coarea formula for $\bar{d}_\Omega$ and exploiting $(\ref{bound_perimeter_unif_eq})$, we get
\begin{equation*}\begin{split}
\Ll_s(\Omega_{-\delta},\Omega\setminus\Omega_{-\delta})&
=\int_{-\delta}^0\Big(\int_{\{\bar{d}_\Omega=\rho\}}\Big(\int_{\Omega_{-\delta}}\frac{dx}{|x-y|^{n+s}}\Big)d\Ha^{n-1}(y)\Big)d\rho\\
&
\leq\int_{-\delta}^0\Big(\int_{\{\bar{d}_\Omega=\rho\}}\Big(\int_{\Co B_{\rho+\delta}(y)}\frac{dx}{|x-y|^{n+s}}\Big)d\Ha^{n-1}(y)\Big)d\rho\\
&
=\frac{n\omega_n}{s}\int_{-\delta}^0\frac{\Ha^{n-1}(\{\bar{d}_\Omega=\rho\})}{(\rho+\delta)^s}\,d\rho\\
&
\leq M \frac{n\omega_n}{s(1-s)}\int_{-\delta}^0\frac{d}{d\rho}(\rho+\delta)^{1-s}\,d\rho=M \frac{n\omega_n}{s(1-s)}\,\delta^{1-s}.
\end{split}\end{equation*}

In the same way we obtain point $(ii)$,
\begin{equation*}\begin{split}
\Ll_s(\Omega_\delta\setminus\Omega,\Omega)&
=\int^\delta_0\Big(\int_{\{\bar{d}_\Omega=\rho\}}\Big(\int_\Omega\frac{dx}{|x-y|^{n+s}}\Big)d\Ha^{n-1}(y)\Big)d\rho\\
&
\leq\int^\delta_0\Big(\int_{\{\bar{d}_\Omega=\rho\}}\Big(\int_{\Co B_\rho(y)}\frac{dx}{|x-y|^{n+s}}\Big)d\Ha^{n-1}(y)\Big)d\rho\\
&
=\frac{n\omega_n}{s}\int^\delta_0\frac{\Ha^{n-1}(\{\bar{d}_\Omega=\rho\})}{\rho^s}\,d\rho\\
&
\leq M \frac{n\omega_n}{s(1-s)}\int^\delta_0\frac{d}{d\rho}\rho^{1-s}\,d\rho=M \frac{n\omega_n}{s(1-s)}\,\delta^{1-s},
\end{split}\end{equation*}
(the other estimate is analogous).
\end{proof}
\end{lem}

\end{subsubsection}

\end{subsection}



\begin{subsection}{(Semi)continuity of the $s$-perimeter}

As shown in Theorem 3.1 of \cite{CRS}, Fatou's Lemma gives the lower semicontinuity of the functional $\Ll_s$.
\begin{prop}\label{semicont_first_prop}
Suppose
\begin{equation*}
A_k\xrightarrow{loc}A\qquad\textrm{and}\qquad B_k\xrightarrow{loc}B.
\end{equation*}
Then
\begin{equation}\label{semicontinuity_first}
\Ll_s(A,B)\leq\liminf_{k\to\infty}\Ll_s(A_k,B_k).
\end{equation}
In particular, if
\begin{equation*}
E_k\xrightarrow{loc}E\qquad\textrm{and}\qquad \Omega_k\xrightarrow{loc}\Omega,
\end{equation*}
then
\begin{equation}
P_s(E,\Omega)\leq\liminf_{k\to\infty}P_s(E_k,\Omega_k).
\end{equation}

\begin{proof}
If the right hand side of $(\ref{semicontinuity_first})$ is infinite, we have nothing to prove, so we can suppose that it is finite.
By definition of the liminf, we can find $k_i\nearrow\infty$ such that
\begin{equation*}
\lim_{i\to\infty}\Ll_s(A_{k_i},B_{k_i})=\liminf_{k\to\infty}\Ll_s(A_k,B_k)=:I.
\end{equation*}
Since $\chi_{A_{k_i}}\to\chi_A$ and $\chi_{B_{k_i}}\to\chi_B$ in
$L^1_{loc}(\R^n)$, up to passing to a subsequence we can suppose that
\begin{equation*}
\chi_{A_{k_i}}\longrightarrow\chi_A\qquad\textrm{and}\qquad\chi_{B_{k_i}}\longrightarrow\chi_B\qquad\textrm{a.e. in }\R^n.
\end{equation*}
Then, since
\begin{equation*}
\Ll_s(A_{k_i},B_{k_i})=\int_{\R^n}\int_{\R^n}\frac{1}{|x-y|^{n+s}}\chi_{A_{k_i}}(x)\chi_{B_{k_i}}(y)\,dx\,dy,
\end{equation*}
Fatou's Lemma gives
\begin{equation*}
\Ll_s(A,B)\leq\liminf_{i\to\infty}\Ll_s(A_{k_i},B_{k_i})=I,
\end{equation*}
proving $(\ref{semicontinuity_first})$.

The second inequality follows just by summing the contributions defining the fractional perimeter.
\end{proof}
\end{prop}

\noindent
Keeping $\Omega$ fixed we obtain Theorem 3.1 of \cite{CRS}.

On the other hand, if we keep the set $E$ fixed and approximate the open set $\Omega$ with a
sequence of open subsets $\Omega_k\subset\Omega$, we get a continuity property.

\begin{prop}\label{continuity_in_open_set_seq}
Let $\Omega\subset\R^n$ be an open set and let $\{\Omega_k\}$ be any sequence of open sets such that
$\Omega_k\xrightarrow{loc}\Omega$.
Then for every set $E\subset\R^n$
\begin{equation*}
P_s(E,\Omega)\leq\liminf_{k\to\infty}P_s(E,\Omega_k).
\end{equation*}
Moreover, if $\Omega_k\subset\Omega$ for every $k$, then
\begin{equation}\label{limit_sub_open}
P_s(E,\Omega)=\lim_{k\to\infty}P_s(E,\Omega_k),
\end{equation}
(finite or not).
\begin{proof}
Since $\Omega_k\xrightarrow{loc}\Omega$,
Proposition \ref{semicont_first_prop} gives the first statement.
Now notice that if $\Omega_k\subset\Omega$, Proposition $\ref{subopensets}$ implies
\begin{equation*}
P_s(E,\Omega_k)\leq P_s(E,\Omega),
\end{equation*}
and hence
\begin{equation*}
\limsup_{k\to\infty}P_s(E,\Omega_k)\leq P_s(E,\Omega),
\end{equation*}
concluding the proof.
\end{proof}
\end{prop}

\begin{rmk}\label{continuity_in_open_set_seq_rmk}
As a consequence, exploiting Corollary $\ref{regular_approx_open_sets_coroll}$, we get
\begin{equation}
P_s(E,\Omega)=\sup_{\Omega'\subsetneq\Omega}P_s(E,\Omega')
=\sup_{\Omega'\subset\subset\Omega}P_s(E,\Omega').
\end{equation}
\end{rmk}

\begin{rmk}\label{counterex_loc_fin_per}
Consider the set $E\subset\R$ constructed in the proof of Example 2.10 in \cite{asymptzero}.
That is, let $\beta_k>0$ be a decreasing sequence such that
\[M:=\sum_{k=1}^\infty\beta_k<\infty\quad\textrm{and}
\quad\sum_{k=1}^\infty\beta_{2k}^{1-s}=\infty,\quad\forall\,s\in(0,1).\]
Then define
\[\sigma_m:=\sum_{k=1}^m\beta_k,\qquad I_m:=(\sigma_m,\sigma_{m+1}),\qquad E:=\bigcup_{j=1}^\infty I_{2j},\]
and let $\Omega:=(0,M)$. As shown in \cite{asymptzero},
\[P_s(E,\Omega)=\infty,\qquad\forall\,s\in(0,1).\]
On the other hand
\[P(E,\Omega')<\infty,\qquad\forall\,\Omega'\subset\subset\Omega,\]
hence $E$ has locally finite $s$-perimeter in $\Omega$, for every $s\in(0,1)$.

Indeed, notice that the intervals $I_{2j}$ accumulate near $M$. Thus, for every $\eps>0$,
all but a finite number of the intervals $I_{2j}$'s fall outside of the open set
$\mathcal O_\eps:=(\eps,M-\eps)$. Therefore
$P(E,\mathcal O_\eps)<\infty$ and hence
\[P_s(E,\mathcal O_\eps)<\infty,\quad\forall\,s\in(0,1).\]
Since $\mathcal O_\eps\nearrow\Omega$ as $\eps\to0^+$, the set $E$ has locally finite $s$-perimeter in $\Omega$
for every $s\in(0,1)$.
\end{rmk}

\begin{prop}\label{subcont_lem_approx}
Let $\Omega\subset\R^n$ be an open set and let $\{E_h\}$ be a sequence of sets such that
\begin{equation*}
E_h\xrightarrow{loc}E\qquad\textrm{and}\qquad \lim_{h\to\infty}P_s(E_h,\Omega)=P_s(E,\Omega)<\infty.
\end{equation*}
Then
\begin{equation}
\lim_{h\to\infty}P_s(E_h,\Omega')=P_s(E,\Omega')\qquad\textrm{for every open set }\Omega'\subset\Omega.
\end{equation}

\begin{proof}
The claim follows from classical properties of limits of sequences.

Indeed, let
\begin{equation*}
\begin{split}
&\qquad\qquad\qquad\qquad\qquad a_h:=P_s(E_h,\Omega'),\\
&
b_h:=\Ll_s\big(E_h\cap(\Omega\setminus\Omega'),\Co E_h\setminus\Omega'\big)
+\Ll_s\big(E_h\setminus\Omega',\Co E_h\cap(\Omega\setminus\Omega')\big),
\end{split}\end{equation*}
and let $a$ and $b$ be the corresponding terms for $E$.\\
Notice that, by Proposition $\ref{subopensets}$, we have
\begin{equation*}
P_s(E_h,\Omega)=a_h+b_h\qquad\textrm{and}\qquad P_s(E,\Omega)=a+b.
\end{equation*}
From Proposition $\ref{semicont_first_prop}$ we have
\begin{equation*}
a\leq\liminf_{h\to\infty}a_h\qquad\textrm{and}\qquad b\leq\liminf_{h\to\infty}b_h,
\end{equation*}
and by hypothesis we know that
\begin{equation*}
\lim_{h\to\infty}(a_h+b_h)=a+b.
\end{equation*}
Therefore
\begin{equation*}
a+b\leq\liminf_{h\to\infty}a_h+\liminf_{h\to\infty}b_h\leq\liminf_{h\to\infty}(a_h+b_h)=a+b,
\end{equation*}
and hence
\begin{equation*}
0\leq\liminf_{h\to\infty}b_h-b=a-\liminf_{h\to\infty}a_h\leq0,
\end{equation*}
so that
\begin{equation*}
a=\liminf_{h\to\infty}a_h\qquad\textrm{and}\qquad b=\liminf_{h\to\infty}b_h.
\end{equation*}
Then, since
\begin{equation*}
\limsup_{h\to\infty}a_h+\liminf_{h\to\infty}b_h\leq\limsup_{h\to\infty}(a_h+b_h)=a+b,
\end{equation*}
we obtain
\begin{equation*}
a=\liminf_{h\to\infty}a_h\leq\limsup_{h\to\infty}a_h\leq a,
\end{equation*}
concluding the proof.
\end{proof}
\end{prop}

\end{subsection}

\begin{subsection}{Compactness}


\begin{prop}[Compactness]\label{compact_prop}
Let $\Omega\subset\R^n$ be an open set.
If $\{E_h\}$ is a sequence of sets such that
\begin{equation}\label{compact_seq_diag_eq}
\limsup_{h\to\infty}P_s^L(E_h,\Omega')\leq c(\Omega')<\infty,\quad\forall\,\Omega'\subset\subset\Omega,
\end{equation}
then there exists a subsequence $\{E_{h_i}\}$ and $E\subset\R^n$ such that
\begin{equation*}
E_{h_i}\cap\Omega\xrightarrow{loc}E\cap\Omega.
\end{equation*}
\begin{proof}
We want to use a compact Sobolev embedding (Corollary 7.2 of \cite{HitGuide}) to construct a limit set via a diagonal argument.

Thanks to Corollary $\ref{regular_approx_open_sets_coroll}$ we know that we can find an increasing sequence of bounded open sets $\{\Omega_k\}$ with smooth boundary such that
\begin{equation*}
\Omega_k\subset\subset\Omega_{k+1}\subset\subset\Omega\qquad\textrm{and}\qquad\bigcup_{k\in\mathbb N}\Omega_k=\Omega.
\end{equation*}
Moreover, hypothesis $(\ref{compact_seq_diag_eq})$ guarantees that
\begin{equation}\label{compact_seq_diag_eq2}
\forall k\quad\exists h(k)\textrm{ s.t.}\quad
P_s^L(E_h,\Omega_k)\leq c_k<\infty,\quad\forall h\geq h(k).
\end{equation}
Clearly
\begin{equation*}
\|\chi_{E_h}\|_{L^1(\Omega_k)}\leq|\Omega_k|<\infty,
\end{equation*}
and hence, since $[\chi_{E_h}]_{W^{s,1}(\Omega_k)}=2P_s^L(E_h,\Omega_k)$,  we have
\begin{equation*}
\|\chi_{E_h}\|_{W^{s,1}(\Omega_k)}\leq c'_k,\quad\forall h\geq h(k).
\end{equation*}
Therefore Corollary 7.2 of \cite{HitGuide} (notice that each $\Omega_k$ is an extension domain) guarantees for every fixed $k$ the existence of a subsequence $h_i\nearrow\infty$ (with $h_1\geq
h(k)$) such that
\begin{equation*}
E_{h_i}\cap\Omega_k\xrightarrow{i\to\infty} E^k
\end{equation*}
in measure, for some set $E^k\subset\Omega_k$.

Applying this argument for $k=1$ we get a subsequence $\{h_i^1\}$ with
\begin{equation*}
E_{h_i^1}\cap\Omega_1\xrightarrow{i\to\infty}E^1.
\end{equation*}
Applying again this argument in $\Omega_2$, with $\{E_{h_i^1}\}$ in place of $\{E_h\}$, we get a subsequence
$\{h_i^2\}$ of $\{h_i^1\}$ with
\begin{equation*}
E_{h_i^2}\cap\Omega_2\xrightarrow{i\to\infty}E^2.
\end{equation*}
Notice that, since $\Omega_1\subset\Omega_2$, we must have $E^1\subset E^2$ in measure (by the uniqueness of the limit in $\Omega_1$). We can also suppose that $h_1^2>h_1^1$.\\
Proceeding inductively in this way we get an increasing subsequence $\{h_1^k\}$ such that
\begin{equation*}
E_{h_1^i}\cap\Omega_k\xrightarrow{i\to\infty}E^k,\qquad\textrm{for every }k\in\mathbb{N},
\end{equation*}
with $E^k\subset E^{k+1}$. Therefore if we define $E:=\bigcup_kE^k$, since $\bigcup_k\Omega_k=\Omega$, we get
\begin{equation*}
E_{h_1^i}\cap\Omega\xrightarrow{loc}E,
\end{equation*}
concluding the proof.
\end{proof}
\end{prop}

\begin{rmk}\label{min_app_seq_rmk}
If $E_h$ is $s$-minimal in $\Omega_k$ for every $h\geq h(k)$, then by minimality we get
\begin{equation*}
P_s^L(E_h,\Omega_k)\leq P_s(E_h,\Omega_k)\leq P_s(E_h\setminus\Omega_k,\Omega_k)\leq P_s(\Omega_k)=:c_k<\infty,
\end{equation*}
since $\Omega_k$ is bounded and has Lipschitz boundary. Therefore $\{E_h\}$ satisfies the hypothesis of Proposition \ref{compact_prop} and we can find a convergent subsequence.
\end{rmk}

\end{subsection}

\end{section}

\begin{section}{Generalized coarea and approximation by smooth sets}

We begin by showing that the $s$-perimeter satisfies a generalized coarea formula (see also \cite{Visintin} and Lemma 10 in \cite{Gamma}).
In the end of this section we will exploit this formula to prove that a set $E$ of locally finite $s$-perimeter can be approximated by smooth sets whose
$s$-perimeter converges to that of $E$.\\

Let $\Omega\subset\R^n$ be an open set. Given a function $u:\R^n\longrightarrow\R$,
we define the functional
\begin{equation}\label{def_ext_func_coarea}
\Fc(u,\Omega):=\frac{1}{2}\int_\Omega\int_\Omega\frac{|u(x)-u(y)|}{|x-y|^{n+s}}dx\,dy
+\int_\Omega\int_{\Co\Omega}\frac{|u(x)-u(y)|}{|x-y|^{n+s}}dx\,dy,
\end{equation}
that is, half the ``$\Omega$-contribution'' to the $W^{s,1}$-seminorm of $u$.\\
Notice that
\begin{equation*}
\Fc(\chi_E,\Omega)=P_s(E,\Omega)
\end{equation*}
and, clearly
\begin{equation*}
\Fc(u,\R^n)=\frac{1}{2}[u]_{W^{s,1}(\R^n)}.
\end{equation*}

\begin{prop}[Coarea]\label{coarea_prop}
Let $\Omega\subset\R^n$ be an open set and let $u:\R^n\longrightarrow\R$.
Then
\begin{equation}\label{coarea_formula}
\Fc(u,\Omega)=\int_{-\infty}^\infty P_s(\{u>t\},\Omega)\,dt.
\end{equation}
In particular
\begin{equation*}
\frac{1}{2}[u]_{W^{s,1}(\Omega)}=\int_{-\infty}^\infty P_s^L(\{u>t\},\Omega)\,dt.
\end{equation*}

\begin{proof}
Notice that for every $x,\, y\in\R^n$ we have
\begin{equation}\label{coarea_first_formula}
|u(x)-u(y)|=\int_{-\infty}^\infty|\chi_{\{u>t\}}(x)-\chi_{\{u>t\}}(y)|\,dt.
\end{equation}
Indeed, the function $t\longmapsto|\chi_{\{u>t\}}(x)-\chi_{\{u>t\}}(y)|$ takes only the values $\{0,1\}$
and it is different from 0 precisely in the interval having $u(x)$ and $u(y)$ as extremes.
Therefore, if we plug $(\ref{coarea_first_formula})$ into $(\ref{def_ext_func_coarea})$ and use Fubini's Theorem, we get
\begin{equation*}
\Fc(u,\Omega)=\int_{-\infty}^\infty\Fc(\chi_{\{u>t\}},\Omega)\,dt=\int_{-\infty}^\infty P_s(\{u>t\},\Omega)\,dt,
\end{equation*}
as wanted.
\end{proof}
\end{prop}

\begin{subsection}{Approximation results for the functional $\Fc$}

In this Section we prove the approximation properties for the functional $\Fc$ which we
need for the proofs of Theorem $\ref{density_smooth_teo}$ and Theorem $\ref{appro_in_bded_open}$.
To this end we consider a (symmetric) mollifier $\eta$, that is
\begin{equation*}
\eta\in C^\infty_c(\R^n),\quad\textrm{supp }\eta\subset B_1,\quad\eta\geq0,\quad\eta(-x)=\eta(x),\quad\int_{\R^n}\eta\,dx=1,
\end{equation*}
and we set
\begin{equation*}
\eta_\eps(x):=\frac{1}{\eps^n}\eta\Big(\frac{x}{\eps}\Big),
\end{equation*}
for every $\eps\in(0,1)$.
Notice that supp $\eta_\eps\subset B_\eps$ and $\int_{\R^n}\eta_\eps=1$.

Given $u\in L^1_{loc}(\R^n)$, we define the $\eps$-regularization of $u$ as the convolution
\begin{equation*}
u_\eps(x):=(u\ast\eta_\eps)(x)=\int_{\R^n}u(x-\xi)\eta_\eps(\xi)\,d\xi,\quad\textrm{for every }x\in\R^n.
\end{equation*}
It is well known that $u_\eps\in C^\infty(\R^n)$ and
\begin{equation*}
u_\eps\longrightarrow u\qquad\textrm{in }L^1_{loc}(\R^n).
\end{equation*}
Moreover, if $u=\chi_E$, then
\begin{equation}\label{forml1_smoothing}
0\leq u_\eps\leq1\qquad\textrm{and}\quad u_\eps(x)=\left\{\begin{array}{cc}1, & \textrm{if }|B_\eps(x)\setminus E|=0\\
0, & \textrm{if }|B_\eps(x)\cap E|=0\end{array}\right.,
\end{equation}
(see e.g. Section 12.3 of \cite{Maggi}).

\begin{lem}\label{dens_lemma}
$(i)\quad$ Let $u\in L^1_{loc}(\R^n)$ and let $\Omega\subset\R^n$ be an open set. Then
\begin{equation}\label{forml_lemma_coarea}
\Fc(u,\Omega)<\infty\quad\Longrightarrow\quad\lim_{\eps\to0^+}\Fc(u_\eps,\Omega')=\Fc(u,\Omega')\qquad\forall\,\Omega'\subset\subset\Omega.
\end{equation}
$(ii)\quad$ Let $u\in W^{s,1}(\R^n)$. Then
\begin{equation*}
\lim_{\eps\to0}[u_\eps]_{W^{s,1}(\R^n)}=[u]_{W^{s,1}(\R^n)}.
\end{equation*}
$(iii)\quad$ Let $u\in W^{s,1}(\R^n)$. Then there exists $\{u_k\}\subset C_c^\infty(\R^n)$ such that
\begin{equation*}
\|u-u_k\|_{L^1(\R^n)}\longrightarrow0\quad\textrm{and}\quad\lim_{k\to\infty}[u_k]_{W^{s,1}(\R^n)}=[u]_{W^{s,1}(\R^n)}.
\end{equation*}
Moreover, if $u=\chi_E$, then $0\leq u_k\leq1$.

\begin{proof}

$(i)\quad$
Given $\mathcal O\subset\R^n$, let $Q(\mathcal O):=\R^{2n}\setminus(\Co\mathcal O)^2$, so that
\begin{equation*}
\Fc(u,\mathcal O)=\frac{1}{2}\int_{Q(\mathcal O)}\frac{|u(x)-u(y)|}{|x-y|^{n+s}}\,dx\,dy.
\end{equation*}
 Notice that if $\mathcal O\subset\Omega$, then
$Q(\mathcal O)\subset Q(\Omega)$
and hence
\begin{equation}\label{forml1_coarea}
\Fc(u,\mathcal O)\leq\Fc(u,\Omega).
\end{equation}
Now let $\Omega'\subset\subset\Omega$ and notice that for $\eps$ small enough we have
\begin{equation}\label{forml2_coarea}
Q(\Omega'-\eps\xi)\subset Q(\Omega)\qquad\textrm{for every }\xi\in B_1.
\end{equation}
As a consequence
\begin{equation}\label{forml4_coarea}
\Fc(u_\eps,\Omega')\leq\int_{B_1}\Fc(u,\Omega'-\eps\xi)\eta(\xi)\,d\xi\leq\Fc(u,\Omega).
\end{equation}
The second inequality follows from $(\ref{forml2_coarea}),\,(\ref{forml1_coarea})$ and $\int_{B_1}\eta=1$.\\
As for the first inequality, we have
\begin{equation*}\begin{split}
\int_{Q(\Omega')}&\frac{|u_\eps(x)-u_\eps(y)|}{|x-y|^{n+s}}dx\,dy\\
&
=\int_{Q(\Omega')}\Big|\int_{\R^n}\big(u(x-\xi)-u(y-\xi)\big)\frac{1}{\eps^n}\eta\Big(\frac{\xi}{\eps}\Big)\,d\xi\Big|\frac{dx\,dy}{|x-y|^{n+s}}\\
&
=\int_{Q(\Omega')}\Big|\int_{B_1}\big(u(x-\eps\xi)-u(y-\eps\xi)\big)\eta(\xi)\,d\xi\Big|\frac{dx\,dy}{|x-y|^{n+s}}\\
&
\leq\int_{B_1}\Big(\int_{Q(\Omega')}\frac{|u(x-\eps\xi)-u(y-\eps\xi)|}{|x-y|^{n+s}}dx\,dy\Big)\eta(\xi)\,d\xi\\
&
=\int_{B_1}\Big(\int_{Q(\Omega'-\eps\xi)}\frac{|u(x)-u(y)|}{|x-y|^{n+s}}dx\,dy\Big)\eta(\xi)\,d\xi.
\end{split}\end{equation*}


We prove something stronger than the claim, that is
\begin{equation}\label{forml_app_98}
\lim_{\eps\to0^+}\Fc(u_\eps-u,\Omega')=0.
\end{equation}
Indeed, notice that
\begin{equation*}
|\Fc(u_\eps,\Omega')-\Fc(u,\Omega')|\leq\Fc(u_\eps-u,\Omega').
\end{equation*}
Let $\psi:\R^{2n}\longrightarrow\R$ be defined as
\begin{equation*}
\psi(x,y):=\frac{u(x)-u(y)}{|x-y|^{n+s}}.
\end{equation*}
Moreover, for every $\eps>0$ and $\xi\in B_1$, we consider the left translation by $\eps(\xi,\xi)$
in $\R^{2n}$, that is
\begin{equation*}
(L_{\eps\xi}f)(x,y):=f(x-\eps\xi,y-\eps\xi),
\end{equation*}
for every $f:\R^{2n}\longrightarrow\R$.\\
Since $\psi\in L^1(Q(\Omega))$, for every $\delta>0$ there exists $\Psi\in C_c^1(Q(\Omega))$ such that
\begin{equation*}
\|\psi-\Psi\|_{L^1(Q(\Omega))}\leq\frac{\delta}{2}.
\end{equation*}
We have
\begin{equation*}\begin{split}
\Fc(u_\eps&-u,\Omega')
=\int_{Q(\Omega')}\frac{|u_\eps(x)-u_\eps(y)-u(x)+u(y)|}{|x-y|^{n+s}}dx\,dy\\
&
\leq\int_{B_1}\Big(\int_{Q(\Omega')}\frac{|u(x-\eps\xi)-u(y-\eps\xi)-u(x)+u(y)|}{|x-y|^{n+s}}dx\,dy\Big)\eta(\xi)\,d\xi\\
&
=\int_{B_1}\|L_{\eps\xi}\psi-\psi\|_{L^1(Q(\Omega'))}\eta(\xi)\,d\xi\\
&
\leq\int_{B_1}\Big(\|L_{\eps\xi}\psi-L_{\eps\xi}\Psi\|_{L^1(Q(\Omega'))}
+\|L_{\eps\xi}\Psi-\Psi\|_{L^1(Q(\Omega'))}\\
&
\qquad\qquad\qquad\qquad
+\|\Psi-\psi\|_{L^1(Q(\Omega'))}\Big)\eta(\xi)\,d\xi.
\end{split}\end{equation*}
Notice that
\begin{equation*}
\|L_{\eps\xi}\psi-L_{\eps\xi}\Psi\|_{L^1(Q(\Omega'))}
=\|\psi-\Psi\|_{L^1(Q(\Omega'-\eps\xi))}
\leq\|\psi-\Psi\|_{L^1(Q(\Omega))}
\end{equation*}
and hence
\begin{equation*}
\Fc(u_\eps-u,\Omega')
\leq\delta+\int_{B_1}\|L_{\eps\xi}\Psi-\Psi\|_{L^1(Q(\Omega'))}\eta(\xi)\,d\xi.
\end{equation*}
For $\eps>0$ small enough we have
\begin{equation*}
\textrm{supp}(L_{\eps\xi}\Psi-\Psi)\subset N_1(\textrm{supp }\Psi)=:K\subset\subset\R^{2n},
\end{equation*}
and
\begin{equation*}
|\Psi(x-\eps\xi,y-\eps\xi)-\Psi(x,y)|\leq2\max_{\textrm{supp }\Psi}|\nabla\Psi|\,\eps.
\end{equation*}
Thus
\begin{equation*}
\int_{B_1}\|L_{\eps\xi}\Psi-\Psi\|_{L^1(Q(\Omega'))}\eta(\xi)\,d\xi
\leq2|K|\max_{\textrm{supp }\Psi}|\nabla\Psi|\,\eps.
\end{equation*}
Passing to the limit as $\eps\to0^+$ then gives
\begin{equation*}
\limsup_{\eps\to0^+}\Fc(u_\eps-u,\Omega')\leq\delta.
\end{equation*}
Since $\delta$ is arbitrary, we get $(\ref{forml_app_98})$.

$(ii)\quad$ Reasoning as above we obtain
\begin{equation*}\begin{split}
\int_{\R^n}&\int_{\R^n}\frac{|u_\eps(x)-u_\eps(y)|}{|x-y|^{n+s}}dx\,dy\\
&
\leq\int_{B_1}\Big(\int_{\R^n}\int_{\R^n}\frac{|u(x-\eps\xi)-u(y-\eps\xi)|}{|x-y|^{n+s}}dx\,dy\Big)\eta(\xi)\,d\xi\\
&
=\int_{B_1}\Big(\int_{\R^n}\int_{\R^n}\frac{|u(x)-u(y)|}{|x-y|^{n+s}}dx\,dy\Big)\eta(\xi)\,d\xi\\
&
=[u]_{W^{s,1}(\R^n)}\int_{B_1}\eta(\xi)\,d\xi,
\end{split}\end{equation*}
that is
\begin{equation}
[u_\eps]_{W^{s,1}(\R^n)}\leq[u]_{W^{s,1}(\R^n)}.
\end{equation}
This and Fatou's Lemma give
\begin{equation*}
[u]_{W^{s,1}(\R^n)}\leq\liminf_{\eps\to0}[u_\eps]_{W^{s,1}(\R^n)}\leq\limsup_{\eps\to0}[u_\eps]_{W^{s,1}(\R^n)}
\leq[u]_{W^{s,1}(\R^n)},
\end{equation*}
concluding the proof.

$(iii)\quad$ The proof is a classical cut-off argument. We consider a sequence of cut-off functions $\psi_k\in C_c^\infty(\R^n)$
such that
\begin{equation*}
0\leq\psi_k\leq1,\quad\textrm{supp }\psi_k\subset B_{k+1}\quad\textrm{and}\quad\psi_k\equiv1\quad\textrm{in }B_k.
\end{equation*}
We can also assume that
\begin{equation*}
\sup_{k\in\mathbb N}|\nabla\psi_k|\leq M_0<\infty.
\end{equation*}
It is enough to show that
\begin{equation}\label{forml_lemma_approx1}
\lim_{k\to\infty}\|u-\psi_k u\|_{L^1(\R^n)}=0\quad\textrm{and}\quad\lim_{k\to\infty}[\psi_k u]_{W^{s,1}(\R^n)}=
[u]_{W^{s,1}(\R^n)}.
\end{equation}

Indeed then we can use $(ii)$ to approximate each $\psi_k u$ with a smooth function $u_k:=(u\psi_k)\ast\eta_{\eps_k}$,
for $\eps_k$ small enough to have
\begin{equation*}
\|\psi_k u-u_k\|_{L^1(\R^n)}<2^{-k}\quad\textrm{and}\quad|[\psi_k u]_{W^{s,1}(\R^n)}-[u_k]_{W^{s,1}(\R^n)}|<2^{-k}.
\end{equation*}
Therefore
\begin{equation*}
\|u-u_k\|_{L^1(\R^n)}
\leq\|u-\psi_k u\|_{L^1(\R^n)}+2^{-k}\longrightarrow0
\end{equation*}
and
\begin{equation*}
|[u]_{W^{s,1}(\R^n)}-[u_k]_{W^{s,1}(\R^n)}|\leq|[u]_{W^{s,1}(\R^n)}-[\psi_k u]_{W^{s,1}(\R^n)}|+2^{-k}\longrightarrow0.
\end{equation*}
Also notice that
\begin{equation*}
\textrm{supp }u_k\subset N_{\eps_k}(\textrm{supp }\psi_k u)\subset B_{k+2}
\end{equation*}
so that $u_k\in C_c^\infty(\R^n)$ for every $k$. Moreover, from the definition of $u_k$ it follows that if $u=\chi_E$, then
$0\leq u_k\leq1$.

For a proof of $(\ref{forml_lemma_approx1})$ see e.g. Lemma 12 in \cite{frac_density}.
\end{proof}
\end{lem}


Now we show that if $\Omega$ is a bounded open set with Lipschitz boundary and if $u=\chi_E$,
then we can find smooth functions $u_h$ such that
\[\Fc(u_h,\Omega)\longrightarrow\Fc(u,\Omega).\]

We first need the following two results.

\begin{lem}\label{stuff_forml1}
Let $\Omega\subset\R^n$ be a bounded open set with Lipschitz boundary. Let $u\in L^\infty(\R^n)$ be such that
$\Fc(u,\Omega)<\infty$. For every $\delta\in(0,r_0)$ let
\begin{equation*}
\varphi_\delta:=1-\chi_{\{|\bar{d}_\Omega|<\delta\}}.
\end{equation*}
Then
\begin{equation}\label{global_bded_conv_l1}
u\varphi_\delta\xrightarrow{\delta\to0}u\quad\textrm{in }L^1(\R^n),
\end{equation}
and
\begin{equation}
\lim_{\delta\searrow0^+}\Fc(u\varphi_\delta,\Omega)=\Fc(u,\Omega).
\end{equation}

\begin{proof}
First of all, notice that
\begin{equation*}
\int_{\R^n}|u\varphi_\delta-u|\,dx=
\int_{\{|\bar{d}_\Omega|<\delta\}}|u|\,dx\leq\|u\|_{L^\infty(\R^n)}\,|\{|\bar{d}_\Omega|<\delta\}|\xrightarrow{\delta\to0}0.
\end{equation*}
Now
\begin{equation*}\begin{split}
\int_\Omega&\int_\Omega\frac{|(u\varphi_\delta)(x)-(u\varphi_\delta)(y)|}{|x-y|^{n+s}}\,dx\,dy\\
&
=
\int_{\Omega_{-\delta}}\int_{\Omega_{-\delta}}\frac{|u(x)-u(y)|}{|x-y|^{n+s}}\,dx\,dy
+
2\int_{\Omega_{-\delta}}\Big(\int_{\Omega\setminus\Omega_{-\delta}}\frac{|u(x)|}{|x-y|^{n+s}}\,dy\Big)dx.
\end{split}\end{equation*}
Since $\Omega_{-\delta}\subset\Omega$, we have
\begin{equation*}
\int_{\Omega_{-\delta}}\int_{\Omega_{-\delta}}\frac{|u(x)-u(y)|}{|x-y|^{n+s}}\,dx\,dy
\leq
\int_\Omega\int_\Omega\frac{|u(x)-u(y)|}{|x-y|^{n+s}}\,dx\,dy.
\end{equation*}
On the other hand, since $|\Omega\setminus\Omega_{-\delta}|\longrightarrow0$, we get
\begin{equation*}
\frac{|u(x)-u(y)|}{|x-y|^{n+s}}\chi_{\Omega_{-\delta}}(x)\chi_{\Omega_{-\delta}}(y)\xrightarrow{\delta\to0}
\frac{|u(x)-u(y)|}{|x-y|^{n+s}}\chi_\Omega(x)\chi_\Omega(y),
\end{equation*}
for a.e. $(x,y)\in\R^n\times\R^n$.

Therefore, by Fatou's Lemma we obtain
\begin{equation}\label{stuff_formul1}
[u]_{W^{s,1}(\Omega)}\leq\liminf_{\delta\searrow0}[u]_{W^{s,1}(\Omega_{-\delta})}
\leq\limsup_{\delta\searrow0}[u]_{W^{s,1}(\Omega_{-\delta})}\leq[u]_{W^{s,1}(\Omega)}.
\end{equation}
Moreover, by point $(i)$ of $(\ref{uniform_bound_strips})$ we get
\begin{equation*}\begin{split}
2\int_{\Omega_{-\delta}}\Big(\int_{\Omega\setminus\Omega_{-\delta}}\frac{|u(x)|}{|x-y|^{n+s}}\,dy\Big)dx&
\leq2\|u\|_{L^\infty(\R^n)}\Ll_s(\Omega_{-\delta},\Omega\setminus\Omega_{-\delta})\\
&
\leq 2C\|u\|_{L^\infty(\R^n)}\,\delta^{1-s}.
\end{split}\end{equation*}
Therefore we find
\begin{equation}
\lim_{\delta\searrow0}[u\varphi_\delta]_{W^{s,1}(\Omega)}=[u]_{W^{s,1}(\Omega)}.
\end{equation}

Now
\begin{equation*}\begin{split}
\int_\Omega&\int_{\Co\Omega}\frac{|(u\varphi_\delta)(x)-(u\varphi_\delta)(y)|}{|x-y|^{n+s}}\,dx\,dy\\
&
=
\int_{\Omega_{-\delta}}\int_{\Co\Omega_\delta}\frac{|u(x)-u(y)|}{|x-y|^{n+s}}\,dx\,dy
+
\int_{\Omega_{-\delta}}\Big(\int_{\Omega_\delta\setminus\Omega}\frac{|u(x)|}{|x-y|^{n+s}}\,dy\Big)dx\\
&
\qquad\qquad+
\int_{\Omega\setminus\Omega_{-\delta}}\Big(\int_{\Co\Omega_\delta}\frac{|u(x)|}{|x-y|^{n+s}}\,dy\Big)dx.
\end{split}\end{equation*}
Since $\Omega_{-\delta}\subset\Omega$ and $\Co\Omega_\delta\subset\Co\Omega$, we have
\begin{equation*}
\int_{\Omega_{-\delta}}\int_{\Co\Omega_\delta}\frac{|u(x)-u(y)|}{|x-y|^{n+s}}\,dx\,dy
\leq
\int_\Omega\int_{\Co\Omega}\frac{|u(x)-u(y)|}{|x-y|^{n+s}}\,dx\,dy.
\end{equation*}
Moreover, since both $|\Omega\setminus\Omega_{-\delta}|\longrightarrow0$ and
$|\Co\Omega\setminus\Co\Omega_\delta|\longrightarrow0$, we have
\begin{equation*}
\frac{|u(x)-u(y)|}{|x-y|^{n+s}}\chi_{\Omega_{-\delta}}(x)\chi_{\Co\Omega_\delta}(y)\xrightarrow{\delta\to0}
\frac{|u(x)-u(y)|}{|x-y|^{n+s}}\chi_\Omega(x)\chi_{\Co\Omega}(y),
\end{equation*}
for a.e. $(x,y)\in\R^n\times\R^n$.

Therefore, again by Fatou's Lemma we obtain
\begin{equation}
\lim_{\delta\searrow0}\int_{\Omega_{-\delta}}\int_{\Co\Omega_\delta}\frac{|u(x)-u(y)|}{|x-y|^{n+s}}\,dx\,dy
=
\int_\Omega\int_{\Co\Omega}\frac{|u(x)-u(y)|}{|x-y|^{n+s}}\,dx\,dy.
\end{equation}

Furthermore, by point $(ii)$ of $(\ref{uniform_bound_strips})$ we get
\begin{equation*}\begin{split}
\int_{\Omega_{-\delta}}\Big(&\int_{\Omega_\delta\setminus\Omega}\frac{|u(x)|}{|x-y|^{n+s}}\,dy\Big)dx
\leq\|u\|_{L^\infty(\R^n)}\Ll_s(\Omega_{-\delta},\Omega_\delta\setminus\Omega)\\
&
\leq\|u\|_{L^\infty(\R^n)}\Ll_s(\Omega,\Omega_\delta\setminus\Omega)
\leq C\|u\|_{L^\infty(\R^n)}\delta^{1-s}
\end{split}\end{equation*}
and also
\begin{equation*}
\int_{\Omega\setminus\Omega_{-\delta}}\Big(\int_{\Co\Omega_\delta}\frac{|u(x)|}{|x-y|^{n+s}}\,dy\Big)dx
\leq C\|u\|_{L^\infty(\R^n)}\delta^{1-s}.
\end{equation*}
Thus
\begin{equation}
\lim_{\delta\searrow0}\int_\Omega\int_{\Co\Omega}\frac{|(u\varphi_\delta)(x)-(u\varphi_\delta)(y)|}{|x-y|^{n+s}}\,dx\,dy
=\int_\Omega\int_{\Co\Omega}\frac{|u(x)-u(y)|}{|x-y|^{n+s}}\,dx\,dy,
\end{equation}
concluding the proof.
\end{proof}
\end{lem}

\begin{lem}\label{second_lemma_for_func_appro}
Let $\Omega\subset\R^n$ be a bounded open set with Lipschitz boundary. Let $v\in L^\infty(\R^n)$ be
such that $\Fc(v,\Omega)<\infty$ and
\begin{equation*}
v\equiv0\quad\textrm{in }\{|\bar{d}_\Omega|<\delta/2\},
\end{equation*}
for some $\delta\in(0,r_0)$. Then
\begin{equation}
\big|\Fc(v,\Omega)-\Fc(v,\Omega_{-\delta/2})\big|\leq C\|v\|_{L^\infty(\R^n)}\delta^{1-s},
\end{equation}
where $C=C(n,s,\Omega)>0$ does not depend on $v$.

\begin{proof}
Since
\begin{equation*}
v\equiv0\quad\textrm{in }\{|\bar{d}_\Omega|<\delta/2\},
\end{equation*}
we have
\begin{equation*}
\Fc(v,\Omega)=\Fc(v,\Omega_{-\delta/2})
+2\int_{\Omega\setminus\Omega_{-\delta/2}}\Big(\int_{\Co\Omega_{\delta/2}}\frac{|v(y)|}{|x-y|^{n+s}}\,dy\Big)dx.
\end{equation*}
Now, by point $(ii)$ of $(\ref{uniform_bound_strips})$ we have
\begin{equation*}\begin{split}
\int_{\Omega\setminus\Omega_{-\delta/2}}\Big(\int_{\Co\Omega_{\delta/2}}\frac{|v(y)|}{|x-y|^{n+s}}\,dy\Big)&
\leq \|v\|_{L^\infty(\R^n)}\Ll_s(\Omega\setminus\Omega_{-\delta/2},\Co\Omega)\\
&
\leq 2^{s-1}C\|v\|_{L^\infty(\R^n)}\,\delta^{1-s}.
\end{split}
\end{equation*}
\end{proof}
\end{lem}

\begin{prop}\label{density_bded_reg_set}
Let $\Omega\subset\R^n$ be a bounded open set with Lipschitz boundary. Let $u\in L^\infty(\R^n)$ be such that
$\Fc(u,\Omega)<\infty$.
Then there exists a sequence $\{u_h\}\subset C^\infty(\R^n)$ such that
\begin{equation}\begin{split}
&(i)\quad \|u_h\|_{L^\infty(\R^n)}\leq\|u\|_{L^\infty(\R^n)},\quad\textrm{and}\quad0\leq u_h\leq1\quad\textrm{if}\quad
0\leq u\leq1,\\
&
(ii)\quad
u_h\xrightarrow{h\to\infty}u\quad\textrm{in }L^1_{loc}(\R^n),\\
&
(iii)\quad\lim_{h\to\infty}\Fc(u_h,\Omega)=\Fc(u,\Omega).
\end{split}\end{equation}

\begin{proof}


By Lemma $\ref{stuff_forml1}$ we know that for every $h\in\mathbb N$ we can find $\delta_h$ small enough such that
\begin{equation}\label{forml_eqtn}
\|u-u\varphi_{\delta_h}\|_{L^1(\R^n)}<2^{-h}\quad\textrm{and}
\quad\big|\Fc(u,\Omega)-\Fc(u\varphi_{\delta_h},\Omega)\big|<2^{-h}.
\end{equation}
We can assume that $\delta_h\searrow0$.

By point $(i)$ of Lemma $\ref{dens_lemma}$ we know that for every $h$ we can find $\eps_h$ small enough such that
\begin{equation}\label{more_formul2}
\|(u\varphi_{\delta_h})\ast\eta_{\eps_h}-u\varphi_{\delta_h}\|_{L^1(B_h)}<2^{-h}
\end{equation}
and
\begin{equation}\label{more_formul1}
\big|\Fc(u\varphi_{\delta_h},\Omega_{-\delta_h/2})-\Fc((u\varphi_{\delta_h})\ast\eta_{\eps_h},\Omega_{-\delta_h/2})\big|<2^{-h}.
\end{equation}
Taking $\eps_h$ small enough, we can also assume that
\begin{equation}\label{more_formul3}
(u\varphi_{\delta_h})\ast\eta_{\eps_h}\equiv0\qquad\textrm{in }\{|\bar{d}_\Omega|<\delta_h/2\},
\end{equation}
since the $\eps$-convolution enlarges the support at most to an $\eps$-neighborhood of the original support.

Let $u_h:=(u\varphi_{\delta_h})\ast\eta_{\eps_h}$. Since we are taking the $\eps_h$-regularization of
the function $u\varphi_{\delta_h}$, which is just a ``rough'' cut-off of $u$, point $(i)$ of our claim is immediate.

By
$(\ref{more_formul2})$
and the first part of $(\ref{forml_eqtn})$ we get point $(ii)$.

As for point $(iii)$, exploiting $(\ref{more_formul3})$ and Lemma \ref{second_lemma_for_func_appro}, we obtain
\begin{equation*}\begin{split}
\big|\Fc(u,\Omega)&-\Fc(u_h,\Omega)\big|\\
&
\leq
\big|\Fc(u,\Omega)-\Fc(u\varphi_{\delta_h},\Omega)\big|+
\big|\Fc(u\varphi_{\delta_h},\Omega)-\Fc(u\varphi_{\delta_h},\Omega_{-\delta_h/2})\big|\\
&
\qquad+\big|\Fc(u\varphi_{\delta_h},\Omega_{-\delta_h/2})-\Fc(u_h,\Omega_{-\delta_h/2})\big|\\
&
\qquad\qquad+\big|\Fc(u_h,\Omega_{-\delta_h/2})-\Fc(u_h,\Omega)\big|\\
&
\leq 2^{-h}+2^sC\|u\|_{L^\infty(\R^n)}\delta_h^{1-s}+2^{-h},
\end{split}\end{equation*}
which goes to 0 as $h\longrightarrow\infty$.
\end{proof}
\end{prop}

\end{subsection}

\begin{subsection}{Proof of Theorem $\ref{density_smooth_teo}$ and Theorem $\ref{appro_in_bded_open}$}

Exploiting Lemma $\ref{dens_lemma}$ and the coarea formula, we can now prove Theorem $\ref{density_smooth_teo}$.

\begin{proof}[Proof of Theorem $\ref{density_smooth_teo}$]
The ``if part'' is trivial. Indeed, just from point $(i)$ and the lower semicontinuity of the $s$-perimeter we get
\begin{equation*}
P_s(E,\Omega')\leq\liminf_{h\to\infty}P_s(E_h,\Omega')<\infty,
\end{equation*}
for every $\Omega'\subset\subset\Omega$.

Now suppose that $E$ has locally finite $s$-perimeter in $\Omega$.\\
The scheme of the proof is similar to that of the classical case (see e.g. the proof of Theorem 13.8 of \cite{Maggi}).

Given a sequence $\eps_h\searrow0^+$ we consider the $\eps_h$-regularization of $u:=\chi_E$ and define
the sets
\begin{equation*}
E_h^t:=\{u_{\eps_h}>t\}\quad\textrm{with }t\in(0,1).
\end{equation*}
Sard's Theorem guarantees that for a.e. $t\in(0,1)$ the sequence $\{E_h^t\}_h$ is made of open sets with smooth boundary.
We will get our sets $E_h$ by opportunely choosing $t$.

Since $u_{\eps_h}\longrightarrow\chi_E$ in $L^1_{loc}(\R^n)$, it is readily seen that for a.e. $t\in(0,1)$
\begin{equation*}
E_h^t\xrightarrow{loc}E,
\end{equation*}
and hence the lower semicontinuity of the $s$-perimeter gives
\begin{equation}\label{forml1_pf_coarea}
P_s(E,\mathcal O)\leq\liminf_{h\to\infty}P_s(E^t_h,\mathcal O),
\end{equation}
for every open set $\mathcal O\subset\R^n$.

Moreover from $(\ref{forml1_smoothing})$ we have
\begin{equation*}
\{0<u_\eps<1\}\subset N_\eps(\partial E)\qquad\forall\,\eps>0,
\end{equation*}
and hence, since $\partial E^t_h\subset\{u_{\eps_h}=t\}$, we obtain
\begin{equation}\label{forml6_pf_coarea}
\partial E_h^t\subset N_{\eps_h}(\partial E),
\end{equation}
which will give $(iii)$ once we choose our $t$.

We improve $(\ref{forml1_pf_coarea})$ by showing that, if $\Omega'\subset\subset\Omega$ is a fixed bounded open set,
then for a.e. $t\in(0,1)$ (with the set of exceptional values of $t$ possibly depending on $\Omega'$),
\begin{equation}\label{forml2_pf_coarea}
P_s(E,\Omega')=\liminf_{h\to\infty}P_s(E^t_h,\Omega').
\end{equation}
 By $(\ref{forml1_pf_coarea})$ and Fatou's Lemma, we have
\begin{equation}\label{forml5_pf_coarea}
P_s(E,\Omega')\leq\int_0^1\liminf_{h\to\infty}P_s(E_h^t,\Omega')\,dt\leq\liminf_{h\to\infty}\int_0^1 P_s(E_h^t,\Omega')\,dt.
\end{equation}
Let $\mathcal O$ be a bounded open set such that $\Omega'\subset\subset\mathcal O\subset\subset\Omega$.
Since $E$ has locally finite $s$-perimeter in $\Omega$, we have $P_s(E,\mathcal O)<\infty$.
Then, since $\Omega'\subset\subset\mathcal O$, point $(i)$ of Lemma $\ref{dens_lemma}$
(with $\mathcal O$ in the place of $\Omega$) implies
\begin{equation}\label{forml3_pf_coarea}
\lim_{h\to\infty}\Fc(u_{\eps_h},\Omega')=\Fc(\chi_E,\Omega')=P_s(E,\Omega').
\end{equation}
Since $0\leq u_{\eps_h}\leq1$, we have $E^t_h=\R^n$ if $t<0$ and $E^t_h=\emptyset$ if $t>1$, and hence
rewriting $(\ref{forml3_pf_coarea})$ exploiting the coarea formula,
\begin{equation*}
\lim_{h\to\infty}\int_0^1P_s(E_h^t,\Omega')\,dt=P_s(E,\Omega').
\end{equation*}
This and $(\ref{forml5_pf_coarea})$ give
\begin{equation*}
\int_0^1\liminf_{h\to\infty}P_s(E_h^t,\Omega')\,dt=P_s(E,\Omega')=\int_0^1 P_s(E,\Omega')\,dt,
\end{equation*}
which implies
\begin{equation}\label{end_approx_proof_eq}
P_s(E,\Omega')=\liminf_{h\to\infty}P_s(E_h^t,\Omega'),\quad\textrm{for a.e. }t\in(0,1),
\end{equation}
as claimed.

Now let the sets $\Omega_k\subset\subset\Omega$ be as in Corollary \ref{regular_approx_open_sets_coroll}.
From \eqref{end_approx_proof_eq} we deduce that for a.e. $t\in(0,1)$ we have
\begin{equation}\label{end_approx_proof_eq2}
P_s(E,\Omega_k)=\liminf_{h\to\infty}P_s(E_h^t,\Omega_k),\qquad\forall\,k\in\mathbb N.
\end{equation}

Therefore, combining all we wrote so far, we find that for a.e. $t\in(0,1)$ the sequence $\{E_h^t\}_h$ is made of open sets with smooth boundary such that $E_h^t\xrightarrow{loc}E$ and both
$(\ref{forml6_pf_coarea})$ and \eqref{end_approx_proof_eq2} hold true.

To conclude, by a diagonal argument we can find $t_0\in(0,1)$ and $h_i\nearrow\infty$ such that,
if we define $E_i:=E^{t_0}_{h_i}$, then
$\{E_i\}$ is a sequence of open sets with smooth boundary
such that $E_i\xrightarrow{loc}E$, with $\partial E_i\subset N_{\eps_{h_i}}(\partial E)$, and
\begin{equation}\label{forml7_pf_coarea}
P_s(E,\Omega_k)=\lim_{i\to\infty}P_s(E_i,\Omega_k),\qquad\forall\,k\in\mathbb N.
\end{equation}

Now notice that if $\Omega'\subset\subset\Omega$, then there exists a $k$ such that $\Omega'\subset\subset\Omega_k$.
Therefore by $(\ref{forml7_pf_coarea})$ and Proposition $\ref{subcont_lem_approx}$ we get $(ii)$.

This concludes the proof of the first part of the claim.\\

Now suppose that $\Omega=\R^n$ and $|E|,\,P_s(E)<\infty$.

Since $|E|<\infty$, we know that $u_\eps\longrightarrow\chi_E$ in $L^1(\R^n)$. Therefore we obtain
$E_h^t\longrightarrow E$ for a.e. $t\in(0,1)$.\\
Moreover, from point $(ii)$ of Lemma $\ref{dens_lemma}$ we know that
\begin{equation*}
\Fc(u,\R^n)<\infty\qquad\Longrightarrow\qquad\lim_{\eps\to0}\Fc(u_\eps,\R^n)=\Fc(u,\R^n).
\end{equation*}
We can thus repeat the proof above and obtain
\begin{equation*}
P_s(E)=\liminf_{h\to\infty}P_s(E_h^t),
\end{equation*}
for a.e. $t\in(0,1)$. For any fixed ``good'' $t_0\in(0,1)$ this directly implies, with no need of a diagonal argument,
the existence of a subsequence $h_i\nearrow\infty$ such that
\begin{equation*}
P_s(E)=\lim_{i\to\infty}P_s(E_{h_i}^{t_0}).
\end{equation*}
We are left to show that in this case we can take the sets $E_h$ to be bounded.

To this end, it is enough to replace the functions $u_{\eps_k}$ with the functions $u_k$
obtained in point $(iii)$ of Lemma $\ref{dens_lemma}$.\\
Indeed, since $u_k$ has compact support, for each $t\in(0,1)$ the set
\begin{equation*}
E_k^t:=\{u_k>t\}
\end{equation*}
is bounded. Since $u_k\longrightarrow u$ in $L^1(\R^n)$ we still find
\begin{equation*}
E_k^t\xrightarrow{loc}E\quad\textrm{for a.e. }t\in(0,1),
\end{equation*}
and, since $0\leq u_k\leq1$ and
\begin{equation*}
\lim_{k\to\infty}\Fc(u_k,\R^n)=P_s(E),
\end{equation*}
we can use again the coarea formula to conclude as above.
\end{proof}



\begin{proof}[Proof of Theorem $\ref{appro_in_bded_open}$]
Exploiting the approximating sequence obtained in Proposition $\ref{density_bded_reg_set}$, we can now prove Theorem
 $\ref{appro_in_bded_open}$ exactly as above.
 
 As for point $(iii)$, recall that the functions $u_h$ of Proposition $\ref{density_bded_reg_set}$ are defined as
 \[u_h=(\chi_E\varphi_{\delta_h})\ast\eta_{\eps_h}.\]
 Notice that, since we can suppose that $\eps_h<\delta_h/2$, we have
 \[u_h=\chi_E\ast\eta_{\eps_h},\qquad\textrm{in }\R^n\setminus N_{2\delta_h}(\partial\Omega).\]
 Therefore, for every $t\in(0,1)$ we find
 \[\partial\{u_h>t\}\subset N_{\eps_h}(\partial E)\subset N_{2\delta_h}(\partial E),\qquad
 \textrm{in }\R^n\setminus N_{2\delta_h}(\partial\Omega).\]
 
 This gives point $(iii)$ once we choose an appropriate $t$, as in the proof of Theorem \ref{density_smooth_teo}.
\end{proof}

\begin{rmk}\label{rmk_conv_every_cpt_subopen}
We remark that by Proposition $\ref{subcont_lem_approx}$ we have also
\begin{equation*}
\lim_{h\to\infty}P_s(E_h,\Omega')=P_s(E,\Omega'),\qquad\textrm{for every }\Omega'\subset\subset\Omega.
\end{equation*}
\end{rmk}

\end{subsection}

\end{section}

\begin{section}{Existence and compactness of $s$-minimal sets}

\begin{subsection}{Proof of Theorem $\ref{confront_min_teo}$}

\begin{proof}[Proof of Theorem $\ref{confront_min_teo}$]
$(i)\Longrightarrow(ii)\quad$ is obvious.

$(ii)\Longrightarrow(iii)\quad$ Let $\Omega'\subset\subset\Omega$ and let $F\subset\R^n$ be such that $F\setminus\Omega'=E\setminus\Omega'$.\\
Since $E\Delta F\subset\Omega'\subset\subset\Omega$, we have
\begin{equation*}
P_s(E,\Omega)\leq P_s(F,\Omega).
\end{equation*}
Then, since $F\setminus\Omega'=E\setminus\Omega'$, by Proposition $\ref{subopensets}$ we get
\begin{equation*}
P_s(E,\Omega')\leq P_s(F,\Omega').
\end{equation*}

$(iii)\Longrightarrow(i)\quad$ Let $E$ be locally $s$-minimal in $\Omega$.

First of all we prove that $P_s(E,\Omega)<\infty$.\\
Indeed, since $E$ is locally $s$-minimal in $\Omega$, in particular it is $s$-minimal in every $\Omega_r$, with $r\in(-r_0,0)$. Thus,
by minimality and $(\ref{unif_bound_lip_frac_per})$, we get
\begin{equation*}
P_s(E,\Omega_r)\leq P_s(E\setminus\Omega_r,\Omega_r)\leq P_s(\Omega_r)\leq M<\infty,
\end{equation*}
for every $r\in(-r_0,0)$.
Therefore 
by $(\ref{limit_sub_open})$ we obtain $P_s(E,\Omega)\leq M$.

Now let $F\subset\R^n$ be such that $F\setminus\Omega=E\setminus\Omega$. Take a sequence $\{r_k\}\subset(-r_0,0)$ such that $r_k\nearrow0$,
let $\Omega_k:=\Omega_{r_k}$, and define
\begin{equation*}
F_k:=(F\cap\Omega_k)\cup(E\setminus\Omega_k).
\end{equation*}
The local minimality of $E$ gives
\begin{equation*}
P_s(E,\Omega_k)\leq P_s(F_k,\Omega_k),\qquad\textrm{for every }k\in\mathbb N,
\end{equation*}
and by $(\ref{limit_sub_open})$ we know that
\begin{equation}
P_s(E,\Omega)=\lim_{k\to\infty}P_s(E,\Omega_k).
\end{equation}
Since $F_k=F$ outside $\Omega\setminus\Omega_k$, and $F_k=E$ in $\Omega\setminus\Omega_k$, we obtain
\begin{equation*}\begin{split}
P_s(F,\Omega_k)&-P_s(F_k,\Omega_k)=\Ll_s(F\cap\Omega_k,\Co F\cap(\Omega\setminus\Omega_k))\\
&
+\Ll_s(\Co F\cap\Omega_k,F\cap(\Omega\setminus\Omega_k))
-\Ll_s(F\cap\Omega_k,\Co E\cap(\Omega\setminus\Omega_k))\\
&\qquad\quad -\Ll_s(\Co F\cap\Omega_k,E\cap(\Omega\setminus\Omega_k)).
\end{split}
\end{equation*}
Notice that each of the four terms in the right hand side is less or equal than $\Ll_s(\Omega_k,\Omega\setminus\Omega_k)$.
Thus
\begin{equation*}
a_k:=|P_s(F,\Omega_k)-P_s(F_k,\Omega_k)|\leq4\,\Ll_s(\Omega_k,\Omega\setminus\Omega_k).
\end{equation*}
Notice that from point $(i)$ of $(\ref{uniform_bound_strips})$ we have $a_k\longrightarrow0$.

Now
\begin{equation*}
P_s(F,\Omega)+a_k\geq P_s(F,\Omega_k)+a_k\geq P_s(F_k,\Omega_k)\geq P_s(E,\Omega_k),
\end{equation*}
and hence, passing to the limit $k\to\infty$, we get
\begin{equation*}
P_s(F,\Omega)\geq P_s(E,\Omega).
\end{equation*}
Since $F$ was an arbitrary competitor for $E$, we see that $E$ is $s$-minimal in $\Omega$.
\end{proof}

\end{subsection}

\begin{subsection}{Compactness}


\begin{proof}[Proof of Theorem $\ref{minimal_comp}$]

Assume $F=E$ outside $\Omega$ and let
\begin{equation*}
F_k:=(F\cap\Omega)\cup(E_k\setminus\Omega).
\end{equation*}
Since $F_k=E_k$ outside $\Omega$ and $E_k$ is $s$-minimal in $\Omega$, we have
\begin{equation*}
P_s(F_k,\Omega)\geq P_s(E_k,\Omega).
\end{equation*}
On the other hand, since $F_k=F$ inside $\Omega$, we have
\begin{equation*}
|P_s(F_k,\Omega)-P_s(F,\Omega)|\leq\Ll_s(\Omega,(F_k\Delta F)\setminus\Omega)=
\Ll_s(\Omega,(E_k\Delta E)\setminus\Omega)=:b_k.
\end{equation*}
Thus
\begin{equation*}
P_s(F,\Omega)+b_k\geq P_s(F_k,\Omega)\geq P_s(E_k,\Omega).
\end{equation*}
If we prove that $b_k\longrightarrow0$, then by lower semicontinuty of the fractional perimeter
\begin{equation}\label{inequality1}
P_s(F,\Omega)\geq\limsup_{k\to\infty}P_s(E_k,\Omega)\geq\liminf_{k\to\infty}P_s(E_k,\Omega)\geq P_s(E,\Omega).
\end{equation}
This shows that $E$ is $s$-minimal in $\Omega$.
Moreover, $(\ref{conv_perimeter})$ follows from $(\ref{inequality1})$ by taking $F=E$.

We are left to show $b_k\longrightarrow0$.\\
Let $r_0$ be as in Proposition $\ref{bound_perimeter_unif}$ and let $R>r_0$. In the end we will let $R\longrightarrow\infty$.
Define
\begin{equation*}
a_k(r):=\Ha^{n-1}\big((E_k\Delta E)\cap\{\bar{d}_\Omega=r\})\big)
\end{equation*}
for every $r\in[0,r_0)$.\\
We split $b_k$ as the sum
\begin{equation*}\begin{split}
b_k&=\Ll_s\big(\Omega,(E_k\Delta E)\cap (\Omega_{r_0}\setminus\Omega)\big)
+\Ll_s\big(\Omega,(E_k\Delta E)\cap (\Omega_R\setminus\Omega_{r_0})\big)\\
&
\qquad\qquad\qquad+\Ll_s\big(\Omega,(E_k\Delta E)\setminus\Omega_R\big).
\end{split}\end{equation*}
Notice that if $x\in\Omega$ and $y\in(\Omega_R\setminus\Omega_{r_0})$, then 
$|x-y|\geq r_0$, and hence
\begin{equation*}\begin{split}
\Ll_s\big(\Omega,(E_k\Delta E)\cap (\Omega_R\setminus\Omega_{r_0})\big)&
=\int_{\Omega_R\setminus\Omega_{r_0}}\chi_{E_k\Delta E}(y)\,dy\int_\Omega\frac{1}{|x-y|^{n+s}}dx\\
&
\leq \frac{|\Omega|}{r_0^{n+s}}|(E_k\Delta E)\cap (\Omega_R\setminus\Omega_{r_0})|.
\end{split}\end{equation*}
Since $E_k\xrightarrow{loc}E$ and $\Omega_R\setminus\Omega_{r_0}$ is bounded, for every fixed $R$ we find
\begin{equation*}
\lim_{k\to\infty}\Ll_s\big(\Omega,(E_k\Delta E)\cap (\Omega_R\setminus\Omega_{r_0})\big)=0.
\end{equation*}
As for the last term, we have
\begin{equation*}
\Ll_s\big(\Omega,(E_k\Delta E)\setminus\Omega_R\big)\leq\Ll_s(\Omega,\Co\Omega_R)\leq
\int_\Omega dx\int_{\Co B_R(x)}\frac{dy}{|x-y|^{n+s}}=\frac{n\omega_n}{s\,R^s}|\Omega|.
\end{equation*}
We are left to estimate the first term. By using the coarea formula, we obtain
\begin{equation*}\begin{split}
\Ll_s\big(\Omega,(E_k&\Delta E)\cap (\Omega_{r_0}\setminus\Omega)\big)\\
&
=\int_0^{r_0}\Big(\int_{\{\bar{d}_\Omega=r\}}\chi_{E_k\Delta E}(y)\Big(\int_\Omega\frac{dx}{|x-y|^{n+s}}\Big)d\Ha^{n-1}(y)\Big)dr\\
&
\leq
\int_0^{r_0}\Big(\int_{\{\bar{d}_\Omega=r\}}\chi_{E_k\Delta E}(y)\Big(\int_{\Co B_r(y)}\frac{dx}{|x-y|^{n+s}}\Big)d\Ha^{n-1}(y)\Big)dr\\
&
=\frac{n\omega_n}{s}\int_0^{r_0}\frac{a_k(r)}{r^s}\,dr.
\end{split}
\end{equation*}
Notice that
\begin{equation*}
\int_0^{r_0}a_k(r)\,dr=|(E_k\Delta E)\cap(\Omega_{r_0}\setminus\Omega)|\xrightarrow{k\to\infty}0,
\end{equation*}
so that
\begin{equation*}
a_k(r)\xrightarrow{k\to\infty}0\qquad\textrm{for a.e. }r\in[0,r_0).
\end{equation*}
Moreover, exploiting $(\ref{bound_perimeter_unif_eq})$ we get
\begin{equation*}
\int_0^{r_0}\frac{a_k(r)}{r^s}\,dr\leq M\int_0^{r_0}\frac{1}{r^s}\,dr=\frac{M}{1-s}r_0^{1-s},
\end{equation*}
and hence, by dominated convergence, we obtain
\begin{equation*}
\lim_{k\to\infty}\int_0^{r_0}\frac{a_k(r)}{r^s}\,dr=0.
\end{equation*}
Therefore
\begin{equation*}
\limsup_{k\to\infty}b_k\leq\frac{n\omega_n}{s}|\Omega|\,R^{-s}.
\end{equation*}
Letting $R\longrightarrow\infty$, we obtain $b_k\longrightarrow0$, concluding the proof.
\end{proof}

\begin{proof}[Proof of Corollary $\ref{local_minima_comp}$]
Let the sets $\Omega_k\subset\subset\Omega$ be as in Corollary
$\ref{regular_approx_open_sets_coroll}$.
By Theorem \ref{minimal_comp} we see that $E$ is $s$-minimal in each $\Omega_k$. Moreover $(\ref{conv_perimeter})$ gives
\begin{equation*}
P_s(E,\Omega_k)=\lim_{h\to\infty}P_s(E_h,\Omega_k),
\end{equation*}
for every $k$. Now if $\Omega'\subset\subset\Omega$, then $\Omega'\subset\Omega_k$ for some $k$.
Thus $E$ is $s$-minimal in $\Omega'$
and we obtain $(\ref{conv_perimeter_locally})$ by Proposition $\ref{subcont_lem_approx}$.
\end{proof}

\end{subsection}

\begin{subsection}{Existence of (locally) $s$-minimal sets}



\begin{proof}[Proof of Theorem \ref{glob_min_exist}]
The ``only if'' part is trivial. Now suppose there exists a competitor for $E_0$ with finite $s$-perimeter in $\Omega$. Then
\begin{equation*}
\inf\{P_s(E,\Omega)\,|\,E\setminus\Omega=E_0\setminus\Omega\}<\infty
\end{equation*}
and we can find a minimizing sequence, that is $\{E_h\}$ with $E_h\setminus\Omega=E_0\setminus\Omega$ and
\begin{equation*}
\lim_{h\to\infty}P_s(E_h,\Omega)=\inf\{P_s(E,\Omega)\,|\,E\setminus\Omega=E_0\setminus\Omega\}.
\end{equation*}
Let $\Omega'\subset\subset\Omega$. Since, for every $h\in\mathbb N$ we have
\begin{equation*}
P_s(E_h,\Omega')\leq P_s(E_h,\Omega)\leq M<\infty,
\end{equation*}
we can use Proposition $\ref{compact_prop}$ to find a set $E'\subset\Omega$ such that
\begin{equation*}
E_h\cap\Omega\xrightarrow{loc}E'
\end{equation*}
(up to subsequence). Since $E_h\setminus\Omega=E_0\setminus\Omega$ for every $h$, if we set $E:=E'\cup(E_0\setminus\Omega)$, then
\begin{equation*}
E_h\xrightarrow{loc}E.
\end{equation*}
The semicontinuity of the fractional perimeter concludes the proof.
\end{proof}

\begin{rmk}\label{rmk_crs_existence}
In particular, if $\Omega$ is a bounded open set with Lipschitz boundary, then (as already proved in \cite{CRS}) we can always find an $s$-minimal set for every $s\in(0,1)$,
no matter what the external data $E_0\setminus\Omega$ is. Indeed in this case
\begin{equation*}
P_s(E_0\setminus\Omega,\Omega)\leq P_s(\Omega)<\infty.
\end{equation*}
Actually, in order to have the existence of $s$-minimal sets for some fixed $s\in(0,1)$, the open set $\Omega$ need not be bounded nor have a regular boundary. It is enough to have
\[P_s(\Omega)<\infty.\]
Then $E_0\setminus\Omega$ has finite $s$-perimeter in $\Omega$ and we can apply Theorem
\ref{glob_min_exist}.
\end{rmk}


\begin{proof}[Proof of Corollary $\ref{loc_min_set_cor}$]
Let the sets $\Omega_k$ be as in Corollary
$\ref{regular_approx_open_sets_coroll}$.\\
From Theorem \ref{glob_min_exist} and Remark \ref{rmk_crs_existence} we know that for every $k$ we can find a set $E_k$ which is $s$-minimal in $\Omega_k$
and such that $E_k\setminus\Omega_k=E_0\setminus\Omega_k$.\\
Notice that, since the sequence $\Omega_k$ is increasing, the set $E_h$ is $s$-minimal in $\Omega_k$ for every $h\geq k$.\\
This gives us a sequence $\{E_h\}$ satisfying the hypothesis of Proposition $\ref{compact_prop}$ (see Remark $\ref{min_app_seq_rmk}$),
and hence (up to a subsequence)
\begin{equation*}
E_h\cap\Omega\xrightarrow{loc}F,
\end{equation*}
for some $F\subset\Omega$. Since $E_h\setminus\Omega=E_0\setminus\Omega$ for every $h$,
if we set $E:=F\cup(E_0\setminus\Omega)$, we obtain
\begin{equation*}
E_h\xrightarrow{loc}E.
\end{equation*}
Theorem $\ref{minimal_comp}$ guarantees that $E$ is $s$-minimal in every $\Omega_k$
and hence also locally $s$-minimal in $\Omega$. Indeed, if $\Omega'\subset\subset\Omega$,
then for some $k$ big enough we have $\Omega'\subset\Omega_k$. Now, since $E$ is $s$-minimal in $\Omega_k$, it is
$s$-minimal also in $\Omega'$.
\end{proof}


\end{subsection}

\begin{subsection}{Locally $s$-minimal sets in cylinders}

Given a bounded open set $\Omega\subset\R^n$, 
we consider
the cylinders
\begin{equation*}
\Omega^k:=\Omega\times(-k,k),\qquad\Omega^\infty:=\Omega\times\R.
\end{equation*}
We recall that, given any set $E_0\subset\R^{n+1}$, by Corollary $\ref{loc_min_set_cor}$ we can find a set $E\subset\R^{n+1}$ which is locally $s$-minimal in
$\Omega^\infty$
 and such that $E\setminus\Omega^\infty=E_0\setminus\Omega^\infty$.

\begin{rmk}\label{rmk_from_compact_to_any_subset}
Actually, if $\Omega$ has Lipschitz boundary then $E$ is $s$-minimal in every cylinder $\mathcal O=\Omega\times(a,b)$ of finite height (notice that $\mathcal O$
is not compactly contained in $\Omega^\infty$).
Indeed, $\mathcal O$ is a bounded open set with Lipschitz boundary and $E$ is locally $s$-minimal in $\mathcal O$.
Thus, by Theorem $\ref{confront_min_teo}$, $E$ is $s$-minimal in $\mathcal O$.\\
As a consequence, $E$ is $s$-minimal in every bounded open subset $\Omega'\subset\Omega$.
\end{rmk}

We are going to consider as exterior data the subgraph
\begin{equation*}
E_0=\Sg(v):=\{(x,t)\in\R^{n+1}\,|\,t<v(x)\},
\end{equation*}
of a function $v:\R^n\longrightarrow\R$, which is locally bounded, i.e.
\begin{equation}\label{locally_bounded_assumption}
M_r:=\sup_{|x|\leq r}|v(x)|<\infty,\qquad\textrm{for every }r>0.
\end{equation}

The following result is an immediate consequence of (the proof of) Lemma 3.3 of \cite{graph}.

\begin{lem}\label{bded_cyl_prop}
Let $\Omega\subset\R^n$ be a bounded open set with $C^{1,1}$ boundary and let $v:\R^n\longrightarrow\R$
be locally bounded. There exists a constant $M=M(n,s,\Omega,v)>0$
such that if $E\subset\R^{n+1}$ is locally $s$-minimal in $\Omega^\infty$,
with $E\setminus\Omega^\infty=\Sg(v)\setminus\Omega^\infty$,
then
\begin{equation*}
\Omega\times(-\infty,-M]\subset E\cap\Omega^\infty\subset\Omega\times(-\infty,M].
\end{equation*}
As a consequence
\begin{equation}\label{formu_trivia2}
E\setminus\big(\Omega\times[-M,M]\big)=\Sg(v)\setminus\big(\Omega\times[-M,M]\big).
\end{equation}
\end{lem}
\begin{proof}
By Remark \ref{rmk_from_compact_to_any_subset}, the set $E$ is $s$-minimal in $\Omega^\infty$ in the sense considered
in \cite{graph}.
Lemma 3.3 of \cite{graph} then guarantees that
\[E\cap\Omega^\infty\subset\Omega\times(-\infty,M].\]
Moreover, the same argument used in the proof shows also that
\[\Co E\cap\Omega^\infty\subset\Omega\times[-M,\infty),\]
(up to considering a bigger $M$).

Since $M>M_{R_0}$, where $R_0$ is such that $\Omega\subset\subset B_{R_0}$,
we get \eqref{formu_trivia2}, concluding the proof.
\end{proof}

Roughly speaking, Lemma \ref{bded_cyl_prop} gives an a priori bound on the variation of
$\partial E$ in the ``vertical'' direction.
In particular, from \eqref{formu_trivia2} we see that it is enough to look for a locally $s$-minimal set among sets which coincide with
$\Sg(v)$ out of $\Omega\times[-M,M]$.

As a consequence, we can prove that a set is locally $s$-minimal in $\Omega^\infty$
if and only if it is $s$-minimal in $\Omega\times[-M,M]$.

\begin{prop}\label{bded_cyl_coroll}
Let $\Omega\subset\R^n$ be a bounded open set with $C^{1,1}$ boundary and let $v:\R^n\longrightarrow\R$ be locally bounded. Let $M$ be as in Lemma \ref{bded_cyl_prop} and
let $k_0$ be the smallest integer $k_0>M$.
Let
$F\subset\R^{n+1}$ be $s$-minimal in $\Omega^{k_0}$,
with respect to the exterior data
\begin{equation}\label{ext_data_eq_cyl}
F\setminus\Omega^{k_0}=\Sg(v)\setminus\Omega^{k_0}.
\end{equation}
Then $F$ is $s$-minimal in $\Omega^k$ for every $k\geq k_0$,
hence is locally $s$-minimal in $\Omega^\infty$.
\end{prop}
\begin{proof}
Let $E\subset\R^{n+1}$ be locally $s$-minimal in $\Omega^\infty$, with respect to the exterior data 
\[E\setminus\Omega^\infty=\Sg(v)\setminus\Omega^\infty.\]
Recall that by Remark \ref{rmk_from_compact_to_any_subset} the set $E$ is $s$-minimal in $\Omega^k$
for every $k$. In particular
\[P_s(E,\Omega^k)<\infty\qquad\forall\,k\in\mathbb N.\]

To prove the Proposition, it is enough to show that
\begin{equation}\label{bded_to_unbded_cyl}
P_s(F,\Omega^k)=P_s(E,\Omega^k),\qquad\textrm{for every }k\geq k_0.
\end{equation}
Indeed, notice that by \eqref{ext_data_eq_cyl} and \eqref{formu_trivia2} we have
\begin{equation}\label{triv_fin1}
F\setminus\Omega^{k_0}=\Sg(v)\setminus\Omega^{k_0}=E\setminus\Omega^{k_0},
\end{equation}
hence, clearly,
\[F\setminus\Omega^k=E\setminus\Omega^k,\qquad\forall\,k\geq k_0.\]
Then, since $E$ is $s$-minimal in $\Omega^k$, from \eqref{bded_to_unbded_cyl}
we conclude that also $F$ is $s$-minimal in $\Omega^k$, for every $k\geq k_0$. In turn, this implies that $F$
is locally $s$-minimal in $\Omega^\infty$.

Exploiting
Proposition $\ref{subopensets}$, by \eqref{triv_fin1} we obtain that for every $k\geq k_0$
\begin{equation}\label{fin_triv2}
P_s(F,\Omega^k)=P_s(F,\Omega^{k_0})+c_k,\qquad P_s(E,\Omega^k)=P_s(E,\Omega^{k_0})+c_k,
\end{equation}
where
\begin{equation*}\begin{split}
c_k=\Ll_s(\Sg(v)&\cap(\Omega^k\setminus\Omega^{k_0}),\Co\Sg(v)\setminus\Omega^{k_0})\\
&
+
\Ll_s(\Sg(v)\setminus\Omega^{k_0},\Co\Sg(v)\cap(\Omega^k\setminus\Omega^{k_0})),
\end{split}
\end{equation*}
which is finite and does not depend on $E$ nor $F$. To see that $c_k$ is finite, simply notice that
\begin{equation*}
c_k\leq P_s(E,\Omega^k)<\infty.
\end{equation*}
Now, by \eqref{triv_fin1} and the minimality of $F$ we have
\begin{equation*}
P_s(F,\Omega^{k_0})\leq P_s(E,\Omega^{k_0}).
\end{equation*}
On the other hand, since also the set $E$ is $s$-minimal in $\Omega^{k_0}$,
again by \eqref{triv_fin1} we get
\begin{equation*}
P_s(E,\Omega^{k_0})\leq P_s(F,\Omega^{k_0}).
\end{equation*}
This and $(\ref{fin_triv2})$ give
\begin{equation*}
P_s(F,\Omega^k)=P_s(F,\Omega^{k_0})+c_k=P_s(E,\Omega^k),
\end{equation*}
proving \eqref{bded_to_unbded_cyl} and concluding the proof.
\end{proof}

It is now natural to wonder whether the set $F$ is actually $s$-minimal in $\Omega^\infty$. The answer, in general, is no.
Indeed, Theorem $\ref{bound_unbound_per_cyl_prop}$ shows
that in general we cannot hope to find an $s$-minimal set in $\Omega^\infty$.

\begin{proof}[Proof of Theorem $\ref{bound_unbound_per_cyl_prop}$]

Notice that by $(\ref{bound_hp_forml_subgraph})$ we have
\begin{equation*}\begin{split}
&E\cap(\Omega^\infty\setminus\Omega^{k+1})=\Omega\times(-\infty,-k-1),\\
&
\Co E\cap(\Omega^\infty\setminus\Omega^{k+1})=\Omega\times(k+1,\infty),
\end{split}\end{equation*}
and
\begin{equation*}
E\cap\Omega^{k+1}\subset\Omega\times(-k-1,k),\qquad\qquad\Co E\cap\Omega^{k+1}
\subset\Omega\times(-k,k+1).
\end{equation*}
Thus
\begin{equation*}\begin{split}
P_s^L(E,\Omega^\infty)&=
P_s^L(E,\Omega^{k+1})+\Ll_s(E\cap(\Omega^\infty\setminus\Omega^{k+1}),\Co E\cap\Omega^{k+1})\\
&\qquad+\Ll_s(\Co E\cap(\Omega^\infty\setminus\Omega^{k+1}),E\cap\Omega^{k+1})
+P_s^L(E,\Omega^\infty\setminus\Omega^{k+1})\\
&
\leq P_s^L(E,\Omega^{k+1})+2\Ll_s(\Omega\times(-\infty,-k-1),\Omega\times(-k,k+1))\\
&\qquad+\Ll_s(\Omega\times(-\infty,-k-1),\Omega\times(k+1,\infty)).
\end{split}
\end{equation*}
Since $d(\Omega\times(-\infty,-k-1),\Omega\times(-k,k+1))=1$, we get
\begin{equation*}\begin{split}
\Ll_s(\Omega\times(-\infty&,-k-1),\Omega\times(-k,k+1))\\
&
\leq\int_{\Omega\times(-k,k+1)}\Big(\int_{\Co B_1(X)}\frac{dY}{|X-Y|^{n+1+s}}\Big)\,dX\\
&
=\frac{(n+1)\omega_{n+1}}{s}(2k+1)|\Omega|.
\end{split}\end{equation*}
As for the last term, since $n+1\geq2$, we have
\begin{equation*}\begin{split}
\Ll_s(\Omega\times(-\infty&,-k-1),\Omega\times(k+1,\infty))\\
&
=\int_\Omega\int_\Omega\Big(\int_{-\infty}^{-k-1}\int_{k+1}^\infty\frac{dt\,d\tau}{(|x-y|^2+(t-\tau)^2)^\frac{n+1+s}{2}}
\Big)dx\,dy\\
&
\leq|\Omega|^2\int_{-\infty}^{-k-1}\Big(\int_{k+1}^\infty\frac{dt}{(t-\tau)^{n+1+s}}\Big)d\tau\\
&
=\frac{|\Omega|^2}{n+s}\int_{-\infty}^{-k-1}\frac{d\tau}{(k+1-\tau)^{n+s}}\\
&
=\frac{|\Omega|^2}{(n+s)(n-1+s)}\,\frac{1}{(2k+2)^{n-1+s}}.
\end{split}\end{equation*}

This shows that $P_s^L(E,\Omega^\infty)<\infty$.\\
Now suppose that $E\subset\R^{n+1}$ satisfies $(\ref{bound_hp_forml_subgraph2})$. Then
\begin{equation*}
P_s^{NL}(E,\Omega^\infty)\geq2\Ll_s(\Omega\times(-\infty,-k),\Co\Omega\times(k,\infty)).
\end{equation*}
Since $\Omega$ is bounded, we can take $R>0$ big enough such that $\Omega\subset\subset B_R$. For every
$T>T_0:=\max\{k,R\}$ we have
\begin{equation*}
\Omega\times(-\infty,-T)\subset\Omega\times(-\infty,-k)\quad\textrm{and}\quad (B_T\setminus B_R)\times(T,\infty)
\subset\Co\Omega\times(k,\infty).
\end{equation*}
Thus for every $T>T_0$
\begin{equation*}\begin{split}
\Ll_s(\Omega&\times(-\infty,-k),\Co\Omega\times(k,\infty))
\geq\Ll_s(\Omega\times(-\infty,-T),(B_T\setminus B_R)\times(T,\infty))\\
&
=\int_\Omega dx\int_{B_T\setminus B_R}dy\int_{-\infty}^{-T}dt
\int_T^\infty\frac{d\tau}{(|x-y|^2+(\tau-t)^2)^\frac{n+1+s}{2}}=:a_T.
\end{split}\end{equation*}
Notice that for every $x\in\Omega,\,y\in B_T\setminus B_R,\,t\in(-\infty,-T)$ and $\tau\in(T,\infty)$, we have
\begin{equation*}
|x-y|\leq|x|+|y|\leq R+T\leq 2T\leq \tau-t,
\end{equation*}
and hence
\begin{equation*}\begin{split}
a_T&\geq\frac{1}{2^\frac{n+1+s}{2}}
\int_\Omega dx\int_{B_T\setminus B_R}dy\int_{-\infty}^{-T}dt
\int_T^\infty\frac{d\tau}{(\tau-t)^{n+1+s}}\\
&
=\frac{|\Omega|}{2^\frac{n+1+s}{2}(n+s)(n-1+s)}\frac{|B_T\setminus B_R|}{(2T)^{n-1+s}}.
\end{split}\end{equation*}
Since $|B_T\setminus B_R|\sim T^n$ as $T\to\infty$, we get $a_T\longrightarrow\infty$. Therefore,
since
\begin{equation*}
P^{NL}_s(E,\Omega^\infty)\geq 2a_T\qquad\textrm{for every }T>T_0,
\end{equation*}
we obtain $P^{NL}_s(E,\Omega^\infty)=\infty$.

To conclude, let $\Omega$ be bounded, with $C^{1,1}$ boundary, and let $v\in L^\infty(\R^n)$. Suppose that there exists a set $E\subset\R^{n+1}$ $s$-minimal in $\Omega^\infty$
with respect to the exterior data $E\setminus\Omega^\infty=\Sg(v)\setminus\Omega^\infty$.\\
Then, thanks to Lemma $\ref{bded_cyl_prop}$, we can find $k$ big enough such that $E$ satisfies
$(\ref{bound_hp_forml_subgraph2})$. 
Since this implies $P_s(E,\Omega^\infty)=\infty$, we reach a contradiction concluding the proof.
\end{proof}

\begin{coroll}\label{non_well_def_frac_area}
In particular
\begin{equation}\label{forml_eq1}
u\in BV_{loc}(\R^n)\cap L^\infty_{loc}(\R^n)\quad\Longrightarrow\quad P_s^L(\Sg(u),\Omega^\infty)<\infty,
\end{equation}
and
\begin{equation}\label{forml_eq2}
u\in
L^\infty(\R^n)\quad\Longrightarrow\quad P_s^{NL}(\Sg(u),\Omega^\infty)=\infty,
\end{equation}
for every bounded open set $\Omega\subset\R^n$.

Furthermore, if $|u|\leq M$ in $\Omega$ and there exists $\Sigma\subset\mathbb S^{n-1}$ with $\Ha^{n-1}(\Sigma)>0$
such that either
\begin{equation*}
u(r\omega)\leq M\quad\textrm{or}\quad u(r\omega)\geq-M\qquad\textrm{for every }\omega\in\Sigma\quad\textrm{and}\quad r\geq r_0,
\end{equation*}
then $P^{NL}_s(\Sg(u),\Omega^\infty)=\infty$.

\begin{proof}
Both $(\ref{forml_eq1})$ and $(\ref{forml_eq2})$ are immediate from Theorem \ref{bound_unbound_per_cyl_prop}, so
we only need to prove the last claim.

Since $\Omega$ is bounded, we can find $R>0$ such that $\Omega\subset\subset B_R$.\\
For every $T>T_0:=\max\{M,R,r_0\}$ define
\begin{equation*}
\mathcal S(T):=\{x=r\omega\in\R^n\,|\,r\in(T_0,T),\,\omega\in\Sigma\}.
\end{equation*}
Notice that $\mathcal S(T)\subset B_T$ and
\begin{equation*}\begin{split}
|\mathcal S(T)|&
=\int_{T_0}^T\Big(\int_{\partial B_r}\chi_{\mathcal S(T)}\,d\Ha^{n-1}\Big)dr
=\int_{T_0}^T\Ha^{n-1}(r\Sigma)\,dr\\
&
=\frac{\Ha^{n-1}(\Sigma)}{n}(T^n-T_0^n).
\end{split}
\end{equation*}
Suppose that $u(r\omega)\leq M$ for every $r\geq r_0$ and $\omega\in\Sigma$. Then, arguing as in the second part of the proof
of Theorem \ref{bound_unbound_per_cyl_prop}, we obtain
\begin{equation*}\begin{split}
P^{NL}_s(\Sg(u),\Omega^\infty)&\geq\Ll_s(\Sg(u)\cap\Omega^\infty,\Co \Sg(u)\setminus\Omega^\infty)\\
&
\geq\Ll_s(\Omega\times(-\infty,-T),\mathcal S(T)\times(T,\infty))\\
&
\geq\frac{|\Omega|}{2^\frac{n+1+s}{2}(n+s)(n-1+s)}\frac{|\mathcal S(T)|}{(2T)^{n-1+s}},
\end{split}
\end{equation*}
for every $T>T_0$.
Since
\begin{equation*}
\frac{|\mathcal S(T)|}{(2T)^{n-1+s}}\sim T^{1-s},
\end{equation*}
which tends to $\infty$ as $T\to\infty$, we get our claim.
\end{proof}
\end{coroll}


In the classical framework, the area functional of a function $u\in C^{0,1}(\R^n)$ is defined as
\begin{equation*}
\mathcal A(u,\Omega):=\int_\Omega\sqrt{1+|\nabla u|^2}\,dx=\Ha^n\big(\{(x,u(x))\in\R^{n+1}\,|\,x\in\Omega\}\big),
\end{equation*}
for any bounded open set $\Omega\subset\R^n$. Exploiting the subgraph of $u$ one then defines the relaxed area
functional of a function $u\in BV_{loc}(\R^n)$ as
\begin{equation}\label{relaxed_area}
\mathcal A(u,\Omega):=P(\Sg(u),\Omega^\infty).
\end{equation}
Notice that when $u$ is Lipschitz the two definitions coincide.

One might then be tempted to define a nonlocal fractional version of the area functional by replacing the 
classical perimeter in $(\ref{relaxed_area})$ with the $s$-perimeter, that is
\begin{equation*}
\mathcal A_s(u,\Omega):=P_s(\Sg(u),\Omega^\infty).
\end{equation*}
However Corollary $\ref{non_well_def_frac_area}$
shows that this definition is ill-posed even for regular functions $u$.\\
On the other hand, it is worth remarking that one could use just the local part of the $s$-perimeter,
but then the resulting functional
\begin{equation*}
\mathcal A_s^L(u,\Omega):=P_s^L(\Sg(u),\Omega^\infty)=\frac{1}{2}[\chi_{\Sg(u)}]_{W^{s,1}(\Omega^\infty)}
\end{equation*}
has a local nature.

Exploiting Theorem 1 of \cite{Davila}, we obtain the following
\begin{lem}
Let $\Omega\subset\R^n$ be a bounded open set with Lipschitz boundary and let $u\in BV(\Omega)\cap L^\infty(\Omega)$.
Then
\begin{equation}
\lim_{s\to1^-}(1-s)\mathcal A^L_s(u,\Omega)=\omega_n\mathcal A(u,\Omega).
\end{equation}

\begin{proof}
Let $k$ be such that $|u|\leq k$. Then $E=\Sg(u)$ satisfies $(\ref{bound_hp_forml_subgraph})$ and hence, arguing as in the beginning of the proof of Theorem $\ref{bound_unbound_per_cyl_prop}$, we
get
\begin{equation*}
\mathcal A_s^L(u,\Omega)=P_s^L(\Sg(u),\Omega^{k+1})+O(1),
\end{equation*}
as $s\to1$.
Since $\Sg(u)$ has finite perimeter in $\Omega^{k+1}$, which is a bounded open set with Lipschitz boundary,
we conclude using Theorem 1 of \cite{Davila} (see also e.g. \cite{mine_fractal} for the asymptotics as $s\to1$ of the $s$-perimeter).\\
Indeed, notice that since $|u|\leq k$, we have
\begin{equation*}
P(\Sg(u),\Omega^{k+1})=P(\Sg(u),\Omega^\infty)=\mathcal A(u,\Omega).
\end{equation*}
\end{proof}
\end{lem}

\end{subsection}

\end{section}

\end{document}